\documentclass[11pt,reqno]{amsart}

\usepackage[hmargin=4cm,vmargin=4cm]{geometry}

\usepackage{amscd,amssymb}
\usepackage{hyperref}

\numberwithin{equation}{section}

\newtheorem {theorem}[equation]                   {Theorem}

\newtheorem {lemma}[equation]           {Lemma}
\newtheorem {proposition}[equation]     {Proposition}

\newtheorem {corollary}[equation]       {Corollary}
\theoremstyle{definition}
\newtheorem {definition}[equation]{Definition}
\newtheorem {remark}[equation]          {Remark}

\newtheorem {example}[equation]         {Example}
\usepackage{xcolor}

\newtheorem{thm}{Theorem}

\def\R{\mathbb{R}}
\newcommand{\pp}[2]{\frac{\partial#1}{\partial#2}}

\begin{document}

\title[Existence and Classification of $b$-contact]{Existence and classification of $b$-contact structures}

\author{Robert Cardona}
\address{Robert Cardona, Departament de Matem\`atiques i Inform\`atica, Universitat de Barcelona, Gran Via, 585, 08007
Barcelona, Spain; Centre de Recerca Matemàtica, Campus de Bellaterra, Edifici C, 08193, Barcelona, Spain.}
\email{robert.cardona@ub.edu}
\thanks{R. Cardona thanks the LabEx IRMIA and the Universit\'e de Strasbourg for their support.}

\author{C\'edric Oms}\address{C\'edric Oms, Basque Center of Applied Mathematics, BCAM, Alameda de Mazarredo, 14, 48009 Bilbao, Bizkaia}\email{coms@bcamath.org}
\thanks{C. Oms acknowledges financial support from the Juan de la Cierva post-doctoral grant (grant number FCJ2021-046811-I). Both authors are partially supported by the ANR grant “CoSyDy” (ANR-CE40-0014), the AEI grant PID2019-103849GB-I00 / AEI / 10.13039/501100011033 and the AGAUR grant 2021 SGR 00603.}

\begin{abstract}
A $b$-contact structure on a $b$-manifold $(M,Z)$ is a Jacobi structure on $M$ satisfying a transversality condition along the hypersurface $Z$. We show that, in three dimensions, $b$-contact structures with overtwisted three-dimensional leaves satisfy an existence $h$-principle that allows prescribing the induced singular foliation. We give a method to classify $b$-contact structures on a given $b$-manifold and use it to give a classification on $S^3$ with either a two-sphere or an unknotted torus as the critical surface. We also discuss generalizations to higher dimensions.
\end{abstract}
\maketitle


\section{Introduction}

\maketitle

Rather new on the mathematical landscape of geometrical structures are the so-called \emph{$b$-contact} structures, introduced in \cite{MO1}. Those structures are the odd-dimensional counterpart to $b$-symplectic structures, which are a class of non-regular Poisson manifolds. The associated Poisson structure of a $b$-symplectic structure is non-degenerate (and thus symplectic) outside a hypersurface, called critical hypersurface, where it degenerates in a transverse manner. The induced Poisson structure on the hypersurface yields a codimension one symplectic foliation. The language of differential forms over the $b$-tangent bundle (first introduced by Melrose \cite{Me} and considered in a symplectic setting in \cite{NT}) is suitable for the description and study of those Poisson structures, called $b$-Poisson structures, as shown by the seminal work of Guillemin-Miranda-Pires \cite{GMP}. With this formalism come many advantages, as for instance Moser's path method, which still holds for $b$-symplectic forms. Since then, $b$-symplectic geometry has become an active field of research. The topology of $b$-symplectic manifolds, also called log-symplectic manifolds, has been studied by several authors \cite{cavalcanti,GMW, FMM, marcut}. In \cite{FMM} the existence problem on open manifolds is addressed using the $h$-principle. Another related class of singular symplectic structures is that of folded symplectic structures \cite{CGW}, whose existence, even on closed manifolds, can be established using the $h$-principle \cite{Can}.

On the contact side of this story, the study of the topological features of $b$-contact structures was initiated in \cite{MO1}, followed by an analysis of their associated Reeb dynamics \cite{MO2,MOP}. Once more, $b$-contact structures on a $b$-manifold $(M,Z)$ are contact outside the hypersurface $Z$, and the contact condition degenerates on the critical hypersurface. A natural way to express this degeneracy is by viewing them as Jacobi structures, which were introduced by Lichnerowicz \cite{Li} and Kirillov \cite{Kir} as generalizations of symplectic, Poisson, and contact structures. In the case of a $b$-contact form, the maximum wedge of the associated Jacobi structure {necessarily satisfies a certain transversality condition (see subsection \ref{subsec:Jacobi})}. The associated singular foliation yields leaves of maximal dimension outside the hypersurface where the induced structure is contact. On the critical hypersurface, an intricate interplay takes place between codimension one locally conformally symplectic leaves and codimension two contact leaves. As it happens with regular symplectic and contact manifolds, $b$-contact structures are induced on certain hypersurfaces of $b$-symplectic manifolds.

In this article, we study the existence and classification of $b$-contact structures on closed orientable three-manifolds. 
 To fix a $b$-manifold, we fix the ambient orientable closed manifold $M$ and the embedded hypersurface $Z$, which is always assumed to be closed but not necessarily connected. Since we will work with orientable $b$-manifolds, this immediately imposes that $Z$ is a separating hypersurface, that is for any connected component $Z'\subset Z$, $M\setminus Z'$ consists of two connected components (see the discussion in Section \ref{ss:bgeo}). Our existence theorems for $b$-contact structures work for a fixed $b$-manifold, which is stronger than the existence results in \cite{MO1}. There, the $b$-manifold cannot be fixed and it is shown that given a closed orientable surface in a closed $3$-manifold {$M$}, there exists a $b$-contact structure having a critical surface given by the union of two diffeomorphic copies of the surface. The second main contribution of this article is that we give a recipe that reduces the classification of $b$-contact structures on any $b$-manifold to ingredients from algebraic and contact topology. We apply this recipe to classify $b$-contact structures on the three-sphere with either a sphere or an unknotted torus as the critical surface. Our classification results for singular geometric structures can be compared with the seminal work of Radko \cite{R}, where a classification of $b$-Poisson structures on closed surfaces was obtained.

Recall that contact structures in dimension three fall into two different classes: overtwisted and tight. The study of the overtwisted ones is well understood. By the classical flexibility result proved by Eliashberg \cite{eli2}, overtwisted contact structures can be understood purely by the homotopy theory of plane fields. In particular, this yields existence and classification results of those structures on closed $3$-manifolds.
Tight contact structures lie on the other side of the spectrum: their study turns out to be way more intricate as they reflect the topology of the underlying manifold. They do not satisfy the $h$-principle and there are examples of $3$-manifolds that don't admit tight contact structures \cite{EH}. The list of manifolds where tight contact structures are classified is short (among others $S^3$, $\mathbb{R}^3$, $S^1\times S^2$ \cite{eli3}). The classification of tight contact structures in manifolds with boundary is for instance settled in the case of the $3$-disk \cite{eli3} or the solid torus \cite{H1,LZ}.

A $b$-contact structure on a manifold $M$ with an orientable critical hypersurface $Z$ is called fully overtwisted if {in each component of $M\setminus Z$ (which correspond to each leaf of maximal dimension of the associated Jacobi structure, see Lemma \ref{lem:dividing=contact})} we get an induced overtwisted contact structure. {Similarly, a $b$-contact structure is tight if in each component of $M\setminus Z$ we get an induced tight contact structure.} We prove that an existence $h$-principle holds for fully overtwisted $b$-contact structures on any fixed orientable $b$-manifold $(M,Z)$, which allows prescribing the singular foliation (as long as this foliation cannot arise for an obvious topological reason). 

{To give a precise statement, let us introduce a bit of terminology. Given a set of separating curves $\Gamma$ in an oriented surface $Z$, let $f$ be a defining function for $\gamma$ (that is $\Gamma$ is the regular zero-level set of $f$). This defines a splitting $Z=Z_+\cup Z_-$ of surfaces with boundary cutting open $Z$ along $\Gamma$, defined as the set where $f$ is non-negative (respectively non-positive). If we consider the defining function $-f$ instead, this leads to the same decomposition of $Z$ where $Z_+$ and $Z_-$ are swapped. For a given set of curves, there are only these two possible decompositions. \textit{From now on we take the following convention:} when we speak of a set of separating curves, we mean the curves together with a choice of $Z_+$ and $Z_-$. Futhermore, given a rank two vector bundle $\eta$ over a three-manifold, we denote by $e(\eta)\in H^2(M,\mathbb{Z})$ the Euler class of $\eta$ {and by $e_Z(\eta)$ its restriction to $H^2(Z,\mathbb{Z})$}. A $b$-isotopy is just an isotopy that preserves $Z$, the natural isotopies on $b$-manifolds, and $\chi$ denotes the Euler characteristic.}

\begin{thm}[Theorem \ref{thm:existencebcontacthprin}]\label{thma}
Let $(M,Z)$ be a closed oriented $b$-manifold of dimension three and let $\eta$ be a $b$-plane field. Then $\eta$ is homotopic through $b$-plane fields to a fully-overtwisted $b$-contact structure whose set of one-dimensional leaves is any collection of separating curves $\Gamma$ in $Z$ such that $\chi(Z_+)-\chi(Z_-)=\langle {e_Z}(\eta), Z\rangle$. Furthermore, if two fully-overtwisted $b$-contact structures have isotopic one-dimensional leaves and are homotopic relative to $Z$, then they are $b$-isotopic.
\end{thm}
 The theorem above admits an interpretation in terms of Jacobi geometry: on a three-manifold $M$ with a given separating hypersurface $Z$, any formal $b$-Jacobi structure (see Definition \ref{def:formalJacobi}) is homotopic to a genuine Jacobi structure $(\Lambda,R)$ such that $\Lambda \wedge R \pitchfork 0$ (i.e. a $b$-Jacobi structure), vanishing exactly along $Z$, with prescribed singular foliation. Indeed, as we will see, the condition on the set of one-dimensional leaves is necessary for the existence of a $b$-contact structure homotopic to $\eta$. 
 
 The proof is based on using the framework of $b$-geometry to apply techniques from contact topology, starting with a semi-local study around the critical surface of the $b$-manifold and an application of Eliashberg's relative $h$-principle for contact structures. Note that since the $b$-manifold is fixed, the singular locus of the $b$-contact structure (and of the associated Jacobi structure) is fixed as well. Closely related results are the existence $h$-principle for $b$-symplectic forms on open manifolds \cite{FMM}, and the existence $h$-principle (on closed manifolds) for folded symplectic forms, a geometric structure that degenerates as well along a critical hypersurface, established by Cannas da Silva \cite{Can}. In both cases, however, the singular locus cannot be prescribed, which contrasts with the results on $b$-contact structures in the present paper.

To the authors' knowledge, this is the first $h$-principle type existence result for Jacobi structures in dimension three that are not transitive (i.e. different from contact structures). {The simplest regular Jacobi structures} are contact structures and conformally symplectic structures, both of which satisfy an existence $h$-principle \cite{bem, BM}. In general, regular Jacobi structures can be interpreted as foliations of constant rank with contact or conformally symplectic leaves. Their flexible behavior has been recently studied by several authors, both in open manifolds \cite{DM, FF} and in closed ones \cite{CPP, bem, GT}. \\

 Using recent advances in the theory of high-dimensional convex hypersurfaces in contact topology, we prove a weaker existence theorem for high-dimensional $b$-contact structures. This one does not allow us to prescribe the singular foliation.
\begin{thm}[Theorem \ref{thm:HDcontactexist}]
{Let $(M,Z)$ be a closed oriented $b$-manifold of dimension $2n+1\geq 3$ such that $Z$ is separating. Then any formal $b$-contact structure is homotopic to a fully overtwisted $b$-contact structure.}
\end{thm}
 In this context, the ``formal" initial data is not just a $b$-hyperplane field, but what we call a formal $b$-contact structure. This formal structure exists on an odd-dimensional $b$-manifold with a separating hypersurface if and only if the manifold is stable almost complex. This implies that any stable almost complex manifold admits a $b$-Jacobi structure with any critical separating hypersurface $Z$. \\

Our first two theorems tackle the existence of $b$-contact structures, and our work continues by studying the classification of $b$-contact structures. This means classifying the singular foliations associated with $b$-Jacobi structures, and for each foliation determining the $b$-isotopy classes of $b$-contact structures. We give a method to achieve such classification results on a given $b$-manifold and apply it in two cases. In the paradigmatic case of $S^3$ with a two-sphere as the singular hypersurface, we obtain the following classification result up to $b$-isotopy.

\begin{thm} [Theorem \ref{thm:bclassificationS3S2}]\label{thm:b}
The classification of positive $b$-contact structures in $(S^3,S^2)$, endowed with some orientation, up to $b$-isotopy is the following.
\begin{itemize}
\item  There exists only one tight $b$-contact structure on $(S^3,S^2)$. {O}n the critical surface there is a single one-dimensional leaf.
\item {$b$-contact structures that are overtwisted in one connected component of $S^3\setminus S^2$ but tight in the other one are determined by specifying which side of the equator is the tight component (i.e. a sign), and by the homotopy class of the $b$-plane field in the closure of the overtwisted component relative to $S^2$ (i.e. by an element in $\mathbb{Z}$: the Hopf invariant)}. On the critical surface, there is a single one-dimensional leaf.
\item  For each $\Gamma$ isotopy class of separating closed curves in $S^2$ satisfying $\chi(S^2_+)=\chi(S^2_-)$, fully overtwisted $b$-contact structures inducing $\Gamma$ as the set of one-dimensional determined by {the homotopy class of the $b$-plane field in the closure of each component of $S^3\setminus S^2$ relative to $S^2$ (i.e. by an element in $\mathbb{Z}^2$: the Hopf invariant of the $b$-plane field in each component)}.
\end{itemize}
\end{thm}
It is implied in the statement that only those isotopy classes of separating curves satisfying $\chi(S^2_+)=\chi(S^2_-)$ can arise as the set of one-dimensional leaves of a $b$-contact structure on $(S^3,S^2)$, and can easily be characterized combinatorically (cf. Lemma \ref{lem:tree}). A classification as above can be carried out as well when the singular locus in $S^3$ is given by an unknotted $2$-torus, see Theorem \ref{thm:classS3T2}. 

The proof of Theorem \ref{thm:b} is obtained by reducing the classification of $b$-contact structures on a given $b$-manifold to algebraic topology and known results in contact topology. Indeed, the recipe of our proof holds for any $b$-manifold $(M,Z)$. The required ingredients to classify $b$-contact structures on a given $b$-manifold are: a classification of the admissible sets of one-dimensional leaves, the classification of relative homotopy classes of plane fields of the connected components of $M\setminus Z$, and the classification of tight contact structures with convex boundary on the connected components of $M\setminus Z$ (viewed as contact manifolds with boundary). In both cases that we study, the classification of tight contact structures of the completion of the connected components of $M \setminus Z$ is available in the literature (see \cite{eli3, H1, LZ}). {However, the classification of tight contact structures in general is still an open question.}\\

Last but not least, our study of the singular hypersurface of a $b$-contact manifold of any odd dimension shows that it always admits a smooth family of contact submanifolds {of codimension two in $M$}. This is interesting on its own since we show that the Reeb vector field of any fixed $b$-contact form restricts to a contact Reeb vector field in each submanifold of this family. As a corollary, we show that the Weinstein conjecture on the existence of periodic Reeb orbits holds for compact $b$-contact manifolds in dimension $5$ along the critical hypersurface. This generalizes a three-dimensional result proved in \cite{MO2} and can potentially give rise to new periodic orbits in problems in celestial mechanics. 

\begin{remark}
 As was communicated to the authors, \'Alvaro del Pino and Aldo Witte also proved the existence of these codimension two contact submanifolds invariant by the $b$-Reeb field of any $b$-contact form (Proposition \ref{prop:codim2contactleaf}). They also show the $\mathbb{R}$-invariance of a $b$-contact structure (Proposition \ref{prop:bconvex}) using a different approach. We would like to thank them for sharing {their results before their paper} \cite{dPW} {was available}.
\end{remark}

\subsection*{Organization of the article}

We start in Section \ref{sec:prel} by giving necessary background information on contact geometry, an introduction to the $b$-tangent bundle and $b$-contact geometry, and $Z$-immersions. In Section \ref{sec:bconv}, we study convexity on $b$-contact structures in greater detail and prove that any $b$-contact structure is convex. Section \ref{sec:fot} contains the proof of Theorem \ref{thma}. The classification of $b$-contact structures is carried out in Section \ref{sec:class}. We finish the article by studying higher dimensional $b$-contact manifolds in Section \ref{sec:Reeb}. We give a weaker existence theorem of fully overtwisted $b$-contact structures in higher dimensions and analyze the existence of periodic Reeb orbits on the critical set.

\subsection*{Acknowledgments}
The authors warmly thank Eva Miranda for proposing the study of the $h$-principle for $b$-contact structures. We would like to thank Alfonso Giuseppe Tortorella and Álvaro del Pino for useful discussions. We are thankful to the anonymous referees for several corrections and suggestions that improved the quality of the paper.

\section{Preliminaries}\label{sec:prel}

\subsection{Contact topology}

A contact structure $\xi$ on an orientable manifold $M$ of dimension $2n+1$ is a nowhere integrable hyperplane field. We assume that the hyperplane field is cooriented so that $\xi$ can be expressed as the kernel of some global one form $\alpha \in \Omega^1(M)$. The non-integrability condition is equivalent to $\alpha \wedge (d \alpha)^n \neq 0$. In three dimensions, an overtwisted disk is a disk embedded in a contact manifold $(M, \xi)$ such that $\xi$ is tangent to the boundary of the disk. {There is also a high-dimensional definition of an overtwisted disk whose precise definition is more involved. We will only say that it is a codimension one-disk with a certain germ of contact structure.} A contact manifold containing an overtwisted disk is called overtwisted, and if such a disk does not exist we say that the contact structure is tight. The dichotomy between overtwisted and tight contact structures has been an important landmark in the field of contact topology. It is known that in all odd dimensions, the high dimensional generalization of overtwisted contact structures satisfy a full $h$-principle \cite{eli2,bem}.

Given a hypersurface $\Sigma$ embedded in a contact manifold $(M,\xi)$, the contact structure induces a singular $1$-dimensional foliation on $\Sigma$, called the characteristic foliation and denoted $\Sigma_\xi$. In three dimensions, it is defined as $\xi\cap T\Sigma$, and the singular points of the foliation are exactly those points $x\in \Sigma$ where $\xi$ is tangent to $T\Sigma$. For more details concerning the characteristic foliation, we refer to \cite[Section 4.6.1]{Ge}. Convex surfaces are a powerful tool to study $3$-dimensional contact manifolds and were introduced in \cite{G1}. The higher dimensional equivalent has been initiated by the work of \cite{HH}, see also \cite{EP}.

An embedded hypersurface $\Sigma \subset (M,\xi)$ in a contact manifold is convex if there exists a contact vector field $X$ (that is a vector field that preserves the contact structure) that is transverse to $\Sigma$. An equivalent characterization is that there is a neighborhood $U\cong \Sigma \times (-\varepsilon,\varepsilon)$ of $\Sigma$ where $\xi$ is defined by a contact form that is invariant with respect to a coordinate in $(-\varepsilon,\varepsilon)$. Such a neighborhood is called $\R$-invariant and the contact form is called $\R$-invariant. In dimension $3$, convex surfaces are $C^\infty$-generic, as shown by Giroux \cite{G1}.
\begin{theorem}\label{thm:giroux}
Let $\Sigma \subset (M,\xi)$ be a surface. Then there exists a $C^\infty$-small isotopy that makes $\Sigma$ convex.
\end{theorem}

The set $\Gamma \subset \Sigma$ defined by $\Gamma:= \{x \in \Sigma \enspace |\enspace X_x \in \xi \}$, where $X$ is the contact vector field transverse to $\Sigma$, is called the \emph{dividing set} of the convex contact surface. It follows from the contact condition that $\Gamma$ is an embedded (closed) hypersurface and that $\Sigma \setminus \Gamma$ is equipped with an exact symplectic form. The dividing set contains the contact topological information of a neighborhood around the convex hypersurface. {Let us finish by pointing out that Theorem \ref{thm:giroux} holds in higher dimensions if we allow the isotopy to be $C^0$-small only \cite{HH, EP}.}





\subsection{$b$-geometry}\label{ss:bgeo}

The language of $b$-differential forms was initiated by Melrose \cite{Me} in his famous work on the Atiyah-Patodi-Singer index theorem, and was later on used in the realm of Poisson geometry by Guillemin-Miranda-Pires \cite{GMP}. We give here a brief outline of the language of $b$-geometry, for more details see \cite{Me,GMP}. Given a hypersurface $Z$, a $b$-vector field is a vector field on $M$ tangent to $Z$. The hypersurface is called the critical hypersurface, or also sometimes the singular locus. The space of $b$-vector fields corresponds to the space of sections of a vector bundle on $M$ called the $b$-tangent bundle and denoted by $^bTM$ {(this is a consequence of the Serre-Swan theorem, see \cite[Section 2.2]{Me})}. It admits an \emph{anchor map} $\pi:{^bTM}\longrightarrow TM$ that is an isomorphism except along $Z$, where it has a one-dimensional kernel. Over $Z$, the line bundle given by this kernel has a canonical non-vanishing section (\cite[Proposition 4]{GMP}). The dual of the $b$-tangent bundle denoted ${^b}T^*M$ is called \emph{$b$-cotangent bundle}. Sections of the bundle $\bigwedge^k {^b}T^*M$ are called $b$-forms and denoted by ${^b}\Omega^k(M)$. Given a local defining function $z$ of the critical set $Z$, that is a function $z:U\to \R$ defined in a neighborhood $U=Z\times (-\varepsilon,\varepsilon)$ of $Z$ whose zero level-set is regular and given by $Z$, a $b$-form $\omega\in {^b}\Omega^k(M)$ can be written in $U$ as
$$\omega=\alpha\wedge\frac{dz}{z}  + \beta,$$
where $\alpha \in \Omega^{k-1}(U)$ and $\beta\in \Omega^k(U)$ are smooth differential forms (in particular $\beta$ does not have a $dz$ term). Any smooth form defines a $b$-form as well. Near $Z$ a smooth form $\gamma$ can be written as $\gamma_1\wedge dz + \gamma_2$, with $\iota_{\frac{\partial}{\partial z}}\gamma_2=0$. Then $\gamma$ corresponds to the $b$-form $z\gamma_1\wedge \frac{dz}{z}+\gamma_2$.

The differential of a $b$-form extends the usual exterior derivative and is defined in a neighborhood of $U$ by
$$d\omega:= d\alpha \wedge \frac{dz}{z} +d\beta.$$
However, observe that in general $d\beta$ is $b$-form, and not a smooth form.  Indeed, if $\beta$ does depend on the $z$ coordinate, the exterior derivative of $\beta$ splits again as $d\beta= d_Z\beta - z \cdot \dot{\beta}\wedge \frac{dz}{z},$ where $\dot{\beta}$ denotes the derivative of $\beta$ with respect to $z$ and $d_Z\beta$ the derivative in the directions tangent to $Z$.

We call a closed coorientable hypersurface $Z$ (possibly with several connected components) that splits $M$ into two (possibly disconnected) disjoint manifolds with common boundary $Z$ a \emph{separating} hypersurface. Given a separating critical hypersurface on an orientable manifold, there exists a non-vanishing $b$-form of maximal degree, a $b$-volume form. {Indeed, as the manifold is orientable, it admits a volume form $\Omega$ and as the critical hypersurface is orientable, it is given by the regular zero-level set of a smooth function $f$. The $b$-volume form is then given by the $b$-form $\frac{\Omega}{f}$.} We say that a manifold admitting a $b$-volume form is an orientable $b$-manifold. Having fixed an orientation of $TM$, in each component of $M\setminus Z$ (where $^bTM$ coincides with $TM$), the $b$-volume form induces either the same or the opposite orientation than the one chosen in $TM$. Hence a way to think of an orientation induced by the $b$-volume form of $(M,Z)$ given by $\frac{1}{f}\Omega$, where $\Omega$ is a volume form on $M$ and $f$ is a smooth function vanishing transversally along each connected component of $Z$. This thus induces a sign on each component of $M\setminus Z$, and every two components that share a common connected component of $Z$ are necessarily of different sign. Since we will always work with an orientable ambient manifold $M$, a $b$-manifold $(M,Z)$ is orientable if and only if $Z$ is separating. 
Furthermore, we adopt the convention that the orientation on the separating hypersurface $Z$ induced by the $b$-volume form is given by the smooth volume form $\Omega_Z$ such that $\frac{dz}{z}\wedge \Omega_Z$ induces the given orientation on the $b$-manifold.

In general, for $Z$ separating, we have a stable isomorphism between $^bTM$ and $TM$, see e.g. \cite[Lemma 2.1]{Can} {and \cite[Proposition 11.1.1]{klaasse}}, which is furthermore canonical up to homotopy \cite{CGW}.

\begin{lemma}\label{lem:stabl}
Let $M$ be a $m$-dimensional manifold with separating hypersurface $Z$. Then there is an isomorphism, canonical up to homotopy, of real vector bundles
$$ TM\oplus \mathbb{R}\cong {}^bTM\oplus \mathbb{R}  $$
\end{lemma}

We will be interested in plane fields in the $b$-tangent bundle of ${^b}TM$.

\begin{definition}\label{def:bplanefield}
A $b$-\emph{hyperplane field} on a $b$-manifold is a co-rank $1$ sub-bundle of the $b$-tangent bundle.

\end{definition}

In the $3$-dimensional set-up, when the critical hypersurface $Z\subset M$ is separating, the previous lemma implies that there always exists $b$-plane fields on $M$.

\begin{proposition}\label{prop:TM=bTM}
Let $Z\subset M^3$ be a separating surface. Then
\begin{enumerate}
\item ${^b}TM$ is isomorphic to $TM$,
\item there exists a non-vanishing section of ${^b}TM$ and therefore also a $b$-plane field on $M$.
\end{enumerate}
\end{proposition}

In particular, in dimension $3$, ${^b}TM$ is parallelizable when $Z$ is separating.
\begin{proof}
As $Z$ is separating, ${^b}TM$ is stably isomorphic to $TM$ by Lemma \ref{lem:stabl}. Thus by Lemma 3.7 in \cite{eli}, ${^b}TM$ is isomorphic to $TM$. As $M$ is of dimension $3$, $\chi(M)=0$ and therefore there exists a non-vanishing section of $TM$. Thus ${^b}TM$ also admits a non-vanishing section. We obtain a $b$-plane field by considering the subbundle orthogonal to such section with respect to any $b$-metric (a bundle metric on $^bTM$).
\end{proof}

In dimension $3$, by Proposition \ref{prop:TM=bTM}, if $Z$ is separating surface then there exists cooriented $b$-plane fields. However, this is not a necessary condition, as the following example shows.

\begin{example}
Consider a three-torus $\mathbb{T}^3$ with coordinates $(x,y,z)$ and choose as $Z$ the $2$-torus $\{\mathbb{T}^2\}\times \{0\}$. The critical surface is not separating, but $^bTM$ clearly admits a coorientable $b$-plane field. Indeed, the $b$-vector field $\pp{}{x}$ is a non-vanishing section of $^bTM$, so its orthogonal complement with respect to any $b$-metric is a $b$-plane field in $^bTM$.
\end{example}

\subsection{Contact geometry on the $b$-tangent bundle}

Contact geometry on $b$-manifolds was introduced in \cite{MO1}, motivated by the symplectic geometry approach to $b$-Poisson structures initiated in \cite{GMP}.

\begin{definition}
A $b$-\emph{contact} form on $(M^{2n+1},Z)$ is a one $b$-form $\alpha \in {^b}\Omega^1(M)$ such that $\alpha \wedge (d\alpha)^n\neq 0$. A $b$-contact structure is defined as the kernel of $b$-contact form.
\end{definition}
A $b$-contact structure is thus a particular case of a $b$-plane field. One could drop the assumption that there exists a globally defined contact form such that $\ker \alpha=\xi$, and just require that this is satisfied locally. For simplicity, we will always assume that a global contact form exists, this means that $\xi$ is coorientable. If we fix an orientation of $(M,Z)$, the $b$-contact structure is said to be positive if $\alpha\wedge (d\alpha)^n$ induces that given orientation. In what follows, we will only consider positive $b$-contact structures. Observe that a $b$-contact structure defines in $M\setminus Z$ a (smooth) contact structure, since a $b$-form is just a smooth form in $M\setminus Z$.\\

Inspired by the definition of convex contact structures, convexity in the $b$-setting can be defined as follows.

	 \begin{definition}\label{def:bconvex}
 A $b$-contact structure $\xi$ is {\emph{convex}} if  there exists a tubular neighborhood $U\cong Z\times (-\varepsilon,\varepsilon)$ of $Z$ and a $b$-contact form on $U$ defining $\xi$ of the form
 $$\alpha= u\frac{dz}{z}+\beta,$$
 with $u \in C^\infty(Z)$ and $\beta \in \Omega^1(Z)$.
 \end{definition}
 
 We will call such a $b$-contact form an $\mathbb{R}$-invariant $b$-contact form. The existence of such a $b$-contact form on some tubular neighborhood of $Z$ is equivalent to the existence, near $Z$, of a $b$-vector field that preserves the $b$-contact {form} and that is transverse (as a section of $^bTM$) to $Z$. {Notice that, whether $\xi$ is convex or not, given a $b$-contact form $\alpha$ the function $u$ restricted to $Z$ is determined (it is a residue in the language of Lie algebroid forms).} As we will see later, it turns out that every $b$-contact structure is convex. The notion of overtwistedness also has its $b$-contact analog.

\begin{definition}\label{def:OT}
A $b$-contact structure is \emph{overtwisted} along a connected component of $M\setminus Z$ if there exists an overtwisted disk embedded in that connected component. It is \emph{fully overtwisted} if every connected component of $M\setminus Z$ admits an embedded overtwisted disk. A $b$-contact structure that does not admit any overtwisted disk in $M\setminus Z$ is called \emph{tight}.
\end{definition}

\subsection{Jacobi structures}\label{subsec:Jacobi}

 In this section, we will shortly discuss the relation between $b$-contact structures and Jacobi structures satisfying a certain transversality condition. This is relevant to identify the classification results that we obtain with classifications of Jacobi structures, see also Remark \ref{rem:definingform} below. A Jacobi structure is a generalization of a Poisson structure, which induces on $M$ a singular foliation by contact and conformally symplectic leaves, and is defined as follows.

\begin{definition}
A \emph{Jacobi structure} on a manifold $M$ is given by a bivector field $\Lambda\in \mathfrak{X}^2(M)$ and a vector field $R\in \mathfrak{X}(M)$ that satisfy 
$$ [\Lambda,\Lambda]=2\Lambda \wedge R \quad \mathcal{L}_R \Lambda=0.$$
\end{definition}

As it is shown in \cite{MO1}, $b$-contact forms can be seen as a particular case of Jacobi structures, where $\Lambda^n \wedge R \pitchfork 0$ (here $(2n+1)$ is the dimension of the manifold). Such a Jacobi structure satisfying that $\Lambda^n \wedge R \pitchfork 0$ is called a \emph{$b$-Jacobi structure}. The following proposition is a global version of \cite[Proposition 6.4]{MO1}. Both results combined yield a global one-to-one correspondence between $b$-Jacobi structures whose singular hypersurface is $Z$, and $b$-contact forms on $(M,Z)$.
\begin{proposition}\label{prop:jacobi}
Let $\Lambda, R$ define a $b$-Jacobi structure, and denote by $Z=(\Lambda^n\wedge R)^{-1}(0)$ the critical hypersurface. Then $(\Lambda,R)$ is induced by a $b$-contact form defined on the $b$-manifold $(M,Z)$.
\end{proposition}

\begin{proof}
By definition, the bi-vector field $\Lambda$ and the vector field $R$ are tangent to the characteristic leaves of the Jacobi structure that they define. In particular they are both sections of $\bigwedge^* {}^bTM$, such that $\Lambda^{n}\wedge R\neq 0$ as sections of the $b$-tangent bundle. The equations 
$$ \alpha(R)=1, \enspace \Lambda(\alpha,\cdot)=0, $$
uniquely determines a section of $^bT^*M$, i.e. a one $b$-form. A classical fact in Jacobi geometry \cite{V} is that in $W=M\setminus Z$, where $^bTM|_{M\setminus Z}\cong TW$, the following relations hold:
$$ \iota_Rd\alpha=0, \enspace \beta= \beta(R)\alpha- d\alpha(\Lambda(\beta,\cdot),\cdot) \text{ for any } \beta\in T^*W. $$
In particular, given any $\gamma \in {}^bT^*M$ such that $\gamma(R)=0$ we have 
$$ \gamma(X)=d\alpha(\Lambda(\gamma,X),X), \enspace \text{for any } X\in {}^bTM$$
at every point $p$ outside $Z$. By smoothness, this holds along $Z$, which shows that $\alpha$ is a $b$-contact form in $(M,Z)$. One easily checks that the Jacobi structure induced by $\alpha$ (see \cite[Proposition 6.4]{MO1}) corresponds to $(\Lambda,R)$.
\end{proof}
\begin{remark}\label{rem:definingform}
The following remark was communicated to us by Alfonso Giuseppe Tortorella. Even if the correspondence above is between pairs $(\Lambda,R)$ and $b$-contact forms, different choices of $b$-contact forms defining the same $b$-contact structure lead to isomorphic {(and even gauge-equivalent)} Jacobi structures (understood as the Jacobi bracket on $C^\infty(M)$ induced by $(R,\Lambda)$). Hence, the classification of $b$-Jacobi structures is equivalent to the classification of $b$-contact structures.
\end{remark}
 A $b$-Jacobi structure, or equivalently, a $b$-contact structure, has the following induced singular foliation. By the transversality condition, the connected components of $M\setminus Z$ are the leaves of maximal dimension, which are (open) contact leaves. On $Z$, we have a foliation by codimension one conformally symplectic leaves and codimension two contact leaves. As we will see in Lemma \ref{lem:dividing=contact}, the contact leaves correspond to the connected components of the dividing set on $Z$ that will be defined in the next section.

In the three-dimensional case, the critical surface of the $b$-Jacobi structure is a $b$-Poisson structure (structures first studied in surfaces by Radko \cite{R}). Indeed, the restriction $\Lambda_Z$ of $\Lambda$ to $Z$ (which is well defined because $\Lambda$ is tangent to $Z$) satisfies $[\Lambda_Z,\Lambda_Z]=0$ for dimensional reasons, so it is Poisson. As furthermore $\Lambda_Z \pitchfork 0$, it is $b$-Poisson.

\section{Revisiting convexity in $b$-contact manifolds}\label{sec:bconv}

The observation that is fundamental for the rest of the work is that any $b$-contact structure is convex as in Definition \ref{def:bconvex}.

 \begin{proposition}\label{prop:bconvex}
 Any $b$-contact structure $\xi=\ker \alpha$ on a $b$-manifold $(M,Z)$ is convex.
 \end{proposition}
 
 \begin{proof}
 In a tubular neighborhood $\mathcal{N}(Z)$ around $Z$, the $b$-contact form decomposes as $\alpha=u \frac{dz}{z}+\beta$ where $u \in C^\infty(\mathcal{N}(Z))$ and $\beta \in \Omega^1(\mathcal{N}(Z))$. The key observation is that along $Z$ the $b$-contact condition only takes into account $u$ and $\beta$ restricted to $Z$. The $b$-contact condition writes
 \begin{equation}\label{eq:bcontact}
  \alpha \wedge (d\alpha)^n= \frac{dz}{z}\wedge (u(d\beta)^n+{n}\beta\wedge du \wedge (d\beta)^{n-1}+z\gamma) \neq 0,
\end{equation}
 where $\gamma \in \Omega^{2n+1}(\mathcal{N}(Z))$. We deduce that 
 
 $i^*(u(d\beta)^n+\beta \wedge du\wedge (d\beta)^{n-1})$ defines a volume form in $Z$, where $i:Z\hookrightarrow \mathcal{N}(Z)$ is the inclusion of $Z$ in its neighborhood. We define the $b$-form $\tilde{\alpha} \in {^b}\Omega^1(\mathcal{N}(Z))$ as
 $$\tilde{\alpha} =i^*u \frac{dz}{z} + i^*\beta.$$
It is of $b$-contact type by the previous discussion, and $\mathbb{R}$-invariant by construction. Consider the family of $b$-forms 
 $$ \alpha_t= (1-t)\alpha + t\tilde \alpha, t\in [0,1].$$
Since $\alpha$ and $\tilde \alpha$ coincide along $Z$, the $b$-form $\alpha_t$ is contact for all $t$ at least in a small neighborhood of $Z$. We can now apply Moser's path method for $b$-forms and Gray's stability for $b$-contact structures \cite[Theorem 5.13]{MO1} to find an isotopy $\varphi_t:\mathcal{N}'(Z)\longrightarrow \mathcal{N}(Z)$ of a perhaps smaller neighborhood of $Z$ such that $\varphi_1^*\xi=\ker \tilde \alpha$, proving the proposition. 
 \end{proof}
 
 \begin{remark}
 This proposition is of course not true for an arbitrary hypersurface on a smooth contact manifold since there exist non-convex hypersurfaces. The difference comes from Equation (\ref{eq:bcontact}): the derivation of $b$-forms imposes that some terms vanish along $Z$, something that does not happen with the usual contact condition near an embedded hypersurface \cite[Section 2.5.4]{Ge}. 
 \end{remark}

We now introduce the notion of dividing set, in analogy with convex surfaces in contact geometry. 
 \begin{definition}\label{def:bdividingset}
The \emph{dividing set} $\Gamma \subset Z$ of a $b$-contact structure $\xi=\ker \alpha$ is defined by the set $\{x\in Z \enspace |\enspace \alpha(v)=0 \}$,  where $v$ any non-vanishing section of the canonical line bundle $\ker \pi|_Z$ over $Z$, where $\pi:{^bTM} \longrightarrow TM$ is the anchor map of the $b$-tangent bundle to $TM$.
\end{definition}

Alternatively, given a splitting as in Definition \ref{def:bconvex}, the dividing set $\Gamma \subset Z$ of a $b$-contact structure $\xi=\ker \alpha$ is defined by  the set $\{x\in Z \enspace | \enspace u|_Z(x)=0 \}$.

 Observe that the dividing set is well-defined, i.e. does not depend on a specific choice of defining $b$-contact form: any other $b$-contact form is of the form $\alpha'=f\alpha$ for a non-vanishing $f\in C^\infty(M)$.

\begin{corollary}
    The diving set of a $b$-contact structure is a non-empty codimension one embedded hypersurface.
\end{corollary}
\begin{proof}
    By Proposition \ref{prop:bconvex}, we can choose an $\mathbb{R}$-invariant contact form in a collar neighborhood of $Z$. It writes $\alpha=u\frac{dz}{z}+\beta$ with $z$ a defining function of $Z$ and $u,\beta$ forms invariant by $z$, and observe that the dividing set is $\Gamma=\{i^*u^{-1}(0)\}$, where $i:Z\rightarrow M$ is the inclusion of the critical set. The contact condition implies that $i^*(u(d\beta)^n+\beta \wedge du\wedge (d\beta)^{n-1})$, and thus whenever $i^*u$ vanishes we necessarily have $di^*u\neq 0$. Furthermore, if $u$ never vanishes then $u(d\beta)^n+n\beta\wedge du\wedge (d\beta)^{n-1}$ defines an exact volume form on $Z$, which is a contradiction with the fact that $Z$ is closed. Hence $\Gamma=u^{-1}(0)$ is non-empty as claimed.
\end{proof}
Furthermore, given a defining convex $b$-contact form defining $\xi$, the dividing set has an induced orientation: indeed, we denote by $Z_\pm:=\{x\in Z| \pm u(x)<0\}$. The induced orientation is given by a vector field on $Z$, pointing transversally out of $Z_+$ along $\Gamma$. As we only consider positive $b$-contact structures, the orientation does not depend on the choice of $b$-contact form defining $\xi$.
As we saw, $b$-contact structures can be seen as particular Jacobi structures. In this formulation, the dividing set coincides with the codimension two leaves of the singular foliation of the associated Jacobi structure. This follows from the local normal forms of $b$-contact forms \cite[Theorem 5.4]{MO1}.

 \begin{lemma}\label{lem:dividing=contact}
     The singular foliation of the $b$-Jacobi structure associated with a $b$-contact form has the following leaves:
     \begin{itemize}
         \item Each connected component of $M\setminus Z$ is a contact leaf of maximal dimension,
         \item each connected component of the dividing set $\Gamma\subset Z$ is a contact leaf of codimension two,
         \item each connected component of $Z\setminus \Gamma$ is a conformally symplectic (which is in fact symplectic) leaf of codimension one.
     \end{itemize}
 \end{lemma}

 \begin{proof}
     The first item is trivial, as the $b$-Jacobi structure at any point in $p\in M\setminus Z$ is of maximal possible rank since the $b$-contact form is contact at $p$. Let $2n+1$ be the dimension of $M$, and $p\in \Gamma$ a point of the dividing set. By the $b$-Darboux theorem \cite[Theorem 5.4]{MO1}, there exists a local coordinate chart $(U,(x_i,y_i,z))$, $i=1,\dots, n$, centered at $p$ such that that the $b$-contact form is given by $\alpha=dx_1+y_1\frac{dz}{z}+\sum_{i=2}^n x_i dy_i$ (the remaining two cases of \cite[Theorem 5.4]{MO1} are excluded by the assumption that $p\in \Gamma$). Observe that $\Gamma \cap U=\{y_1=z=0\}$. A direct computation yields that the bi-vector field of the Jacobi structure in these coordinates is given by 
     $$\Lambda=\frac{\partial}{\partial y_1}\wedge z\frac{\partial}{\partial z}-y_1\frac{\partial}{\partial y_1}\wedge \frac{\partial}{\partial x_1} + \sum_{i=2}^n (y_i \pp{}{x_1}-\pp{}{x_i} )\wedge \pp{}{y_i}.$$
     The rank of $\Lambda$ is thus $2n-2$ along $\Gamma \cap U$, meaning that the Jacobi structure spans a leaf of the associated foliation through $p$ is of dimension $2n-1$ with an induced contact structure. This corresponds to the connected component of $\Gamma$ where $p$ lies.

     Let now $p$ be a point in $Z\setminus \Gamma$. Then we are in the two remaining cases of the $b$-Darboux theorem \cite[Theorem 5.4]{MO1}. Denote by $R_\alpha$ the Reeb $b$-vector field of $\alpha$.
     \begin{equation*}
        \alpha|_p=\begin{cases}
             \left(dx_1+(1+y_1)\frac{dz}{z}+\sum_{i=2}^n x_i dy_i\right)|_p, \quad \text{if $\pi(R_\alpha|_{p})\neq 0$}\\
             \left(\frac{dz}{z}+\sum_{i=1}^nx_idy_i\right)|_p,\quad \text{otherwise,}
         \end{cases}
     \end{equation*}
          where $\pi:{^b}TM\to TM$ is the anchor map. The associated Jacobi structures are
    $$
             \left(\frac{\partial}{\partial y_1}\wedge z\frac{\partial}{\partial z}-(1+y_1)\frac{\partial}{\partial y_1}\wedge \frac{\partial}{\partial x_1} + \sum_{i=2}^n (y_i \pp{}{x_1}-\pp{}{x_i} )\wedge \pp{}{y_i}\right)|_p,$$
     if  $\pi((R_\alpha)_{p})\neq 0$, and 
     $$\left(\sum_{i=1}^n \pp{}{x_i}\wedge \pp{}{y_i}+ x_i z\pp{}{z}\wedge \pp{}{x_i}\right)|_p,$$
     otherwise. Furthermore, in the first case $R_{\alpha}=\pp{}{x_1}$ and in the second case $R_{\alpha}$ vanishes along $TZ$ as a $b$-vector field, and thus vanishes in $TM|_p$ as a $b$-vector field. Since $\Lambda$ has rank $2n$ but $\Lambda^n \wedge R_{\alpha}=0$ at these points, $p$ belongs to a codimension two conformally symplectic leaf. The fact that the leaf is symplectic follows from observing that $\Lambda$ on $Z\setminus \Gamma$ is induced by the restriction of the form $d\beta$ to $Z$, where $\alpha=u\frac{dz}{z}+\beta$, which is non-degenerate away from $\Gamma$.
 \end{proof}
Notice that in particular the singular foliation only depends on the $b$-contact structure, and not on the defining form. This is not surprising in view of Remark \ref{rem:definingform}.

We can now introduce the concept of overtwistedness for $b$-contact structures, by looking at each leaf of maximal dimension.
 
 The next lemma shows that in a neighborhood of the critical set, a convex $b$-contact structure can be viewed as a contact structure. This lemma has been previously used in \cite{MO1} for the singularization of convex contact structures and desingularization of $b$-contact structures. Recall that a $b$-isotopy of $(M,Z)$ is an isotopy of $M$ preserving $Z$.

 \begin{lemma}\label{lem:convexbcontact=convexcontact}
 Let $U\cong Z\times (-\varepsilon,\varepsilon)$ be a tubular neighborhood of $Z$. There is a one-to-one correspondence between $\mathbb{R}$-invariant $b$-contact forms in $(U,Z)$ and $\mathbb{R}$-invariant contact forms in $U$.
 \end{lemma}
 \begin{proof}
Any $\mathbb{R}$-invariant $b$-contact form writes $\alpha=u\frac{dz}{z}+\beta$, where $z$ is a coordinate in $(-\varepsilon,\varepsilon)$ and $u\in C^\infty(Z)$ and $\beta \in \Omega^1(Z)$. The correspondence is hence given by the identification $dz \longmapsto \frac{dz}{z}$, since for $\mathbb{R}$-invariant $u,\beta$, the $b$-contact condition and the contact condition are the same.
 \end{proof}
Intuitively, the lemma asserts that results known for convex contact hypersurfaces can be translated to $b$-contact structures.

\begin{remark} \label{rem:charfoldef}
Let $\alpha$ be the convex $b$-contact form such that $\xi = \ker \alpha$ and we write $\alpha= u\frac{dz}{z}+\beta$ where $u \in C^\infty(Z)$ and $\beta \in \Omega^1(Z)$. In a tubular neighborhood of $Z$, we can define the contact structure defined by the kernel of $\overline{\alpha}=udz+\beta$. The characteristic foliation of $Z$ (in the $b$-contact sense) could be defined by the characteristic foliation (in the usual sense) of $\ker \overline{\alpha}$. However, in contrast to the contact case, this definition depends on the choice of $b$-contact form defining the $b$-contact structure. For this reason, in what follows, we will work exclusively with the dividing associated with the $b$-contact structure.
 \end{remark}
 
 Up to here, our discussion holds for $b$-contact manifolds of any dimension. The following corollaries, which only make sense for three-dimensional $b$-contact manifolds, will be of interest for our classification results.\\

\paragraph{\textbf{Giroux criteria for three dimensional $b$-contact manifolds.}}
 
 In this section, we restrict to three dimensions. A key lemma to relate results in convex surface theory in the contact case with the $b$-contact setting is the following.
\begin{lemma}\label{lem:convexdelta}
Let $\xi$ be a convex $b$-contact structure on an odd-dimensional $b$-manifold $(M, \hat Z)$, and let $Z$ be a connected component of $\hat Z$ with dividing set $\Gamma\subset Z$ its dividing set. Let $U\cong Z\times (-\varepsilon,\varepsilon)$ be a tubular neighborhood of $Z$ where $\xi=\ker \alpha$ for a $\mathbb{R}$-invariant $b$-contact form $\alpha$. For any surface $Z^*=Z\times \{\pm \delta\}$ with $0<\delta<\varepsilon$, let $(N,\xi|_N)$ be the open contact manifold corresponding to the connected component of $M\setminus Z$ where $Z^*$ lies. Then $Z^*$ is a (usual) convex surface in $(N,\xi|_N)$ with dividing set $\Gamma \times \{\delta\}$. Furthermore
$(V,\xi|_V)=(N\cap U,\xi_{N\cap U})$ is contactomorphic to a contact $\mathbb{R}$-invariant neighborhood of $Z^*$ inside $N$.
\end{lemma}
\begin{proof}
     Assume $\delta>0$ for simplicity, the other case being analogous. Let $z$ denote the coordinate in a tubular neighborhood $U\cong Z\times (-\varepsilon,\varepsilon)$ around $Z$ for which the $b$-contact structure is defined by a $b$-form $\alpha$ that is $\mathbb{R}$-invariant.  The fact that $Z^*$ is convex in $(N,\xi|_N)$ is not immediate from the existence of the $\mathbb{R}$-invariant $b$-contact form $\alpha$, since it is not $\mathbb{R}$-invariant when understood as a smooth form. Indeed, it writes $\alpha=\frac{u}{z}dz+\beta$ and notice that although $u$ is $\mathbb{R}$-invariant, the function $\frac{u}{z}$ is not $z$-independent. However, one way to see that $Z^*$ is convex is by observing that the vector field $Y=z\pp{}{z}$ defines, in a neighborhood $Z\times (\delta-\tau, \delta+\tau)$ of $Z^*$ completely contained in $N$, a contact vector field transverse to $Z^*$. Indeed, we have 
\begin{align*}
    \mathcal{L}_Y\alpha&= d\iota_Y\alpha + \iota_Yd\alpha=du-du=0.\\
\end{align*}  Observe that the dividing set (in the usual sense of contact geometry) of the convex surface $t=\delta$ is given by $\Gamma\times \{\delta\}$, since $\alpha(Y)=0$ exactly when $u=0$, and hence $\Gamma=\{p\in Z^* \enspace |\enspace Y_p\in \xi_p\}$. 

Furthermore, assuming $\delta$ is small enough, there exists some $\delta' \in (\delta, \varepsilon)$ and time $t^*>0$ such that the flow $\varphi_t$ of the transverse contact vector field $Y=z\pp{}{z}$ satisfies $\varphi_{t^*}(Z\times (0,\delta'))=Z\times(0,\epsilon)$. This contactomorphism identifies $(V,\xi_V)$ 
with a contact $\mathbb{R}$-invariant neighborhood of $Z^*$, see e.g. the proof of \cite[Lemma 4.6.19]{Ge}. 
\end{proof}

 The first corollary is the well-known Giroux criteria for contact structures on three-manifolds adapted to $b$-contact structures on $b$-manifolds of dimension three.
 
 \begin{corollary}[Giroux criteria for $b$-contact structures]\label{cor:bgirouxcrit}
Let $\xi$ be a convex $b$-contact structure on a three-dimensional $b$-manifold $(M, \hat Z)$. If the dividing set $\Gamma$ on a connected component $Z$ of $\hat Z$ contains a circle $\gamma$ contractible in $Z$, then in a neighborhood $U$ of $Z$, each connected component of $U\setminus Z$ contains an overtwisted disk compactly supported in $U\setminus Z$, unless $Z\cong S^2$ and $\Gamma$ is connected.
\end{corollary}

\begin{proof}
 Let $U\cong Z\times (-\varepsilon,\varepsilon)$ be a neighborhood of $Z$ with a $\mathbb{R}$-invariant $b$-contact form $\alpha$. Let $(V,\xi|_V)$ be a connected component of $U\setminus Z$, understood as a contact manifold. By Lemma \ref{lem:convexdelta}, it is contactomorphic to a neighborhood of the convex surface $Z^*=Z\times \{\delta\}$ with dividing set $\Gamma\times \{\delta\}$. An application of the usual Giroux criterion for convex surfaces in contact manifolds tells us that in a small neighborhood of $Z^*$, strictly contained in $V$, there is an overtwisted disk if and only if $\Gamma$ contains a circle $\gamma$ contractible in $Z$ and we are not in the case that $Z\cong S^2$ and $\Gamma$ is connected.
\end{proof}

Another important fact asserts that the dividing set of the $b$-contact structure determines the $b$-contact structure in a neighborhood around the critical set $Z$. 
\begin{proposition}\label{cor:dividingset}
Let $\xi_0,\xi_1$ be two $b$-contact structures on $(M^3,Z)$ with dividing sets $\Gamma_0,\Gamma_1$. Assume that there is an orientation preserving isotopy from $\Gamma_0$ to $\Gamma_1$. Then there is a $b$-isotopy $\psi_t : \mathcal{N}(Z) \to \mathcal{N}(Z)$, $t \in [0,1]$, where $\mathcal{N}(Z)$ is a tubular neighbourhood around $Z$ such that $\psi_1(\xi_0)=\xi_1$.
\end{proposition}

\begin{proof}
First, we can assume without loss of generality that the two $b$-contact structures have the same dividing set. Indeed, the isotopy from $\Gamma_1$ to $\Gamma_2$ can be extended to a $b$-isotopy $\Psi_t:\mathcal{N}(Z)\to \mathcal{N}(Z)$. The $b$-contact structure given by $d\Psi_1(\xi_0)$ and $\xi_1$ then have the same dividing set.

Let  $\alpha_i$, $i=0,1$ be two convex $b$-contact forms defining the $b$-contact structure (which exist due to Proposition \ref{prop:bconvex}), given in the neighborhood by
$$\alpha_i=u_i \frac{dz}{z}+\beta_i,$$
where $u_i\in C^\infty(Z)$ and $\beta_i\in \Omega^1(Z)$.

The proof of the proposition follows the proof of \cite[Theorem 4.8.11]{Ge}. In the tubular neighborhood, we define the following two smooth contact forms $\overline{\alpha}_i=u_idz+\beta_i$. Both contact forms define the same dividing set, and the characteristic foliation (as for smooth contact forms) defined by $\ker \beta_i$ is transverse to the dividing set.
More precisely, as in \cite[Theorem 4.8.11]{Ge}, we may assume that we have area forms $\Omega_i\in \Omega^2(Z)$ and vector fields $X_i$ on $Z$ ($i=0,1$), such that $Z_\pm=\{x\in Z|\pm \mathrm{div}_{\Omega_i} X_i>0\}$ and the divergence is zero on $\Gamma$. As $\Omega_i$ are two area forms, there exists a positive function $g: Z \to \mathbb{R}^+$ such that $\Omega_0=g\Omega_1$.
As $g\mathrm{div}_{g\Omega}(X)=\mathrm{div}_\Omega(gX)$ for $g:\Sigma \to \mathbb{R}^+$, we know that for $X_0'=gX_0$, $\pm\mathrm{div}_{\Omega_1}(X_1')>0$ on $Z_\pm$. Abusing notation, we will thus denote by $X_1$ the vector field $gX_1$ and by $\Omega$ the volume form $\Omega_1$. Consider the family of vector fields $X_t=tX_0+(1-t)X_1$, for $t\in [0,1]$. These vector fields will likewise satisfy the divergence condition $\pm \mathrm{div}_\Omega (X_t) > 0$ on $Z_\pm$. We thus find the family of contact forms $\overline{\alpha}_t=u_t dz+\iota_{X_t}\Omega$. Similarily, we define the path of $b$-contact forms by $\alpha_t:=u_t\frac{dz}{z}+\iota_{X_t}\Omega$. We apply Moser's path method for this family of $b$-contact forms (see \cite[Theorem 5.11]{MO1}) and get thus an isotopy $\psi_t:\mathcal{N}(Z) \to \mathcal{N}(Z)$ that preserves $Z$ and such that $d\psi_t (\ker \alpha_1)=\ker \alpha_t$.
\end{proof}

\section{Fully overtwisted $b$-contact structures}\label{sec:fot}

In this section, we will adapt the $h$-principle for overtwisted contact structures \cite{eli2} to a class of $b$-contact structures on three-manifolds. In Section \ref{sec:Reeb}, we will give a weaker statement for higher dimensional manifolds.

\subsection{Formal $b$-contact and $b$-Jacobi structures}

In order to prove an $h$-principle for $b$-contact structures, we first define the formal counterpart of $b$-contact structures, and its interpretation in terms of Jacobi geometry.

\begin{definition}
Let $(M,Z)$ be a $b$-manifold of dimension $2n+1$. A pair $(\alpha,\omega)$ of $b$-forms $\alpha \in {^b}\Omega^1(M)$ and $\omega \in {^b}\Omega^2(M)$ is called an \emph{almost $b$-contact structure} if $\alpha\wedge \omega^n >0$.
\end{definition}
 By analogy with the contact case and other settings where $h$-principles are studied, we will sometimes refer to an almost $b$-contact structure as a ``formal" $b$-contact structure, as it plays the role of a formal solution to our problem. An almost $b$-contact structure is hence equivalent to a corank $1$ subbundle in $^bTM$ equipped with a symplectic bundle structure. When the ambient manifold is three-dimensional, an almost $b$-contact structure is simply a coorientable $b$-plane field since any rank two vector bundle admits a homotopically unique fiberwise symplectic structure (or equivalently a fiberwise complex structure). By duality, there is also a formal notion of $b$-Jacobi structure in terms of multivector fields.

\begin{definition}\label{def:formalJacobi}
A \emph{formal $b$-Jacobi structure} on a $b$-manifold $(M,Z)$ of dimension $(2n+1)$ is a couple $(\Lambda,E)$ where $\Lambda$ is a $b$-bivector field and $E$ is a $b$-vector field such that $\Lambda^n \wedge E \neq 0$.
\end{definition}
Since there is the anchor map from the $b$-tangent bundle to the tangent bundle, one can look at a formal $b$-Jacobi structure directly in the smooth tangent bundle. In this case, it corresponds to a vector field $E'$ and bivector field $\Lambda'$ of the tangent bundle that are tangent along $Z$ (and hence define sections of the $b$-tangent bundle) and satisfy $E'\wedge (\Lambda')^n\pitchfork 0$ and $(E'\wedge (\Lambda')^n)^{-1}(0)=Z$. We point out that requiring only this transversality condition does not imply that $E'$ and $\Lambda'$ define sections of the $b$-tangent bundle, so this condition needs to be imposed. This condition is not new, it is already required for the same reasons when considering formal $b$-Poisson structures, as in \cite[Section 3]{FMM}. In the following lemma, we show that the two definitions are equivalent. 

\begin{lemma}
Formal $b$-Jacobi structures are in correspondence with formal $b$-contact structures.
\end{lemma}
\begin{proof} 
The correspondence between formal $b$-Jacobi structures and formal $b$-contact structures is given by duality on the $b$-tangent bundle, similarly to Proposition \ref{prop:jacobi} and \cite[Proposition 6.4]{MO1}. Indeed, let $(\alpha,\omega)$ be a formal $b$-contact structure in $(M,Z)$. Define $E$ as the vector field satisfying $\alpha(E)=1$ and $\iota_E\omega=0$. The $b$-multivector field $\Lambda$ is determined by $\Lambda (df,dg)= \omega(X_f, X_g)$, where given a function $h$ the vector field $X_h$ is determined by the equations
$$\begin{cases} \alpha(X_h)=h\\ \iota_{X_h}\omega= dh(E)\alpha - dh\end{cases} $$
Conversely, a formal $b$-Jacobi structure $(\Lambda, E)$ determines a formal $b$-contact structure. The one $b$-form $\alpha$ is determined by $\alpha(E)=1$ and $\Lambda(\alpha,\cdot)=0$, and the two $b$-form is determined by the equations $\iota_E\omega=0$ and $\omega(\Lambda(\gamma,\cdot),\cdot)=\gamma$ for any $\gamma \in {^b}TM$ such that $\gamma(E)=1$.
\end{proof}

We can now settle the existence of a formal $b$-contact structure, which only depends on the topology of the ambient manifold.

\begin{proposition}\label{prop:formalb}
$M$ admits a formal $b$-Jacobi structure on $(M,Z)$, where $Z$ is a separating hypersurface, if and only if $M$ is stable almost complex.
\end{proposition}

\begin{proof}
First, it is easy to argue the well-known fact that $M$ is stable almost complex if and only if $M\times \mathbb{R}$ admits an almost complex structure. We include it for completeness, as we couldn't find the complete argument written down in detail in the literature. Let $2m-1$ be the dimension of $M$. An almost complex structure in $TM\times \mathbb{R} \times \mathbb{C}$ is the same as the existence of a trivial section of the associated $SO(2m+2)/U(m+1)$ bundle. Here we look at the bundle over $M$, but it is equivalent to looking at it as a bundle over $M\times \mathbb{R}$ since this space retracts to $M$. The obstruction to finding such a section lies in $\pi_{q} (SO(2m+2)/U(m+1)$ for $q$ up to $2m-2$ (one less of the dimension of $M$). However, in the stable range of the gamma groups $SO/U$ we have that for a fixed $q$ and variable $n$, the groups $\pi_q(SO(2n)/U(n)$ are all isomorphic for all $n$ as long as $q<2n-1$. In particular, we have $\pi_{q}(SO(2m+2)/U(m+1))\cong \pi_{q}(SO(2m)/U(m))$ for each $q=0,...,2m-2$. So the obstructions vanish either for both bundles or for none of them, so there exists an almost complex structure in $TM\times \mathbb{R} \times \mathbb{C}$ if and only if there exists one in $TM\times \mathbb{R}$. We conclude by applying this observation inductively to $TM\oplus \mathbb{R} \oplus \mathbb{C}^k$.\\

Now, assume that $M$ admits a formal $b$-Jacobi structure $(R,\Lambda)$. Consider the $b$-manifold $(M\times \mathbb{R}, Z\times \mathbb{R})$, denote by $t$ the $\mathbb{R}$ coordinate and by $\pi:M\times \mathbb{R} \longrightarrow M$ the projection to the first factor. The bivector field
$$ \Pi= \pi^*R\wedge \pp{}{t}+ \pi^*\Lambda $$ 
is a non-degenerate section of $\Lambda^2(^bT(M\times \mathbb{R}))$ that induces a fiberwise complex structure on $^bT(M\times \mathbb{R})$. As vector bundles over $M\times \mathbb{R}$ or over $M$, we have $^bT(M\times \mathbb{R})\cong {}^bTM\oplus \mathbb{R} \cong TM \oplus \mathbb{R}$. We deduce that the latter admits an almost complex structure.\\
 
For the other implication, if $TM\times \mathbb{R}$ admits an almost complex structure then so does {$^bT(M\times \mathbb{R})$, where $^bT(M\times \mathbb{R})$ denotes the $b$-tangent bundle of the $b$-manifold $(M\times \mathbb{R}, Z\times \mathbb{R})$.} This is equivalent to having a non-degenerate two $b$-form $\omega$. Then $(M\times \{0\}, Z\times \{0\})$ is a $b$-submanifold of $(M\times \mathbb{R}, Z\times \mathbb{R})$ {(meaning that $M\times \{0\}$ is a submanifold of $M\times \mathbb{R}$ transverse to the critical set $Z\times \mathbb{R}$)}. The inclusion $i:M\times \{0\}\hookrightarrow M\times \mathbb{R}$ induces a canonical map
$$^bi: {}^bTM \longrightarrow {}^bT(M\times \mathbb{R}),$$ 
and $^bi^*\omega$ defines a $b$-form of degree $2$ of the $b$-manifold $(M,Z)$ which is of maximal rank in $M$. In particular, its kernel is a one-dimensional line bundle of $^bTM$. Choosing any one $b$-form $\alpha$ that is positive along the kernel of $^bi^*\omega$, we obtain a formal $b$-contact structure $(\alpha,{}^bi^*\omega)$ and hence a formal $b$-Jacobi structure.
\end{proof}

 \begin{remark}
    Alternatively, Proposition \ref{prop:formalb} can be proved using formal $b$-contact structures and Lemma \ref{lem:stabl}. Given a formal $b$-contact structure $(\alpha,\omega)$ on $M$, the Reeb vector field $R\in \Gamma({^b}TM)$ is defined by the equations  $\alpha(R)=1,\iota_R\omega=0$. As $\omega$ is non-degenerate on $\xi=\ker \alpha$, we can define an almost complex structure on ${^b}TM\oplus \mathbb{R}_t$ by defining it on the splitting $\xi \oplus \langle R,\partial_t\rangle$. By Lemma \ref{lem:stabl}, this thus shows that $M$ is stable almost complex. Conversely, given an almost complex structure on $TM\oplus \mathbb{R}$, and by Lemma \ref{lem:stabl} on ${^b}TM\oplus \mathbb{R}$, we get a non-degenerate $2$-form $\omega\in {^b}\Omega^2(M\times \mathbb{R})$. The $b$-forms $\alpha:=i^*\iota_{\partial_t} \omega$ and $\beta:=i^*\omega$ define a formal $b$-contact structure on $M\times\{0\}\cong M$, where $i:M\times \{0\} \rightarrow M\times \mathbb{R}$ denotes the inclusion of the zero section.
\end{remark}

\subsection{Homological obstruction of dividing sets}
 Given a formal $b$-contact structure $(\tilde{\alpha},\tilde{\omega})$ on the $b$-manifold $(M,Z)$, we want to show that is homotopic to a genuine $b$-contact form $ \alpha$, through formal $b$-contact structures. As a $b$-contact structure \emph{prints} a dividing set on the critical hypersurface, the natural question arises of which sets can be realized as the dividing set of $\ker \alpha$. Therefore, in this subsection, we discuss the homological conditions that the dividing set of a $b$-contact structure has to satisfy. We will then prove in the following section that given a $b$-plane field and any multicurve $\Gamma$ on $Z$ satisfying these necessary homological conditions, the $b$-plane field can be homotoped (through $b$-plane fields) to a $b$-contact structure $\xi=\ker\alpha$ such that $\Gamma$ is the dividing set of $\xi$ along $Z$. 
\begin{remark}\label{rem:plane}
    In the three-dimensional case, proving that a formal $b$-contact structure is homotopic to a genuine $b$-contact {structure} is equivalent to proving that a coorientable $b$-plane field is homotopic to a $b$-contact structure through $b$-plane fields. This is because the set of almost $b$-contact structures defining a plane field is contractible. First, given a $b$-plane field $\eta$ we can always choose a one $b$-form $\alpha$ such that $\ker \alpha=\eta$ (and this choice is unique up to homotopy of non-vanishing one-forms defining $\eta$). We want to prove that the set of $b$-forms of degree $2$ that are non-degenerate along $\eta$ is contractible. To see this, recall first that a fiberwise non-degenerate two-form on $\eta$ determines an almost complex structure on it.  Furthermore, the set of fiberwise non-degenerate two-forms on $\eta$ adapted to any given almost complex structures is also contractible, see for instance \cite[Proposition 2.6.4 (ii)]{McS}. Extending a fiberwise non-degenerate two-form on $\eta$ to a $b$-form of degree $2$ of maximal rank in ${^b} TM$ amounts to specifying its kernel on ${^b}TM$, a choice that is unique up to homotopy. Thus, it is enough to show that the set of fiberwise almost complex structures in $\eta$ is contractible. This is easily seen as follows. An almost complex structure on $\eta$ corresponds to a section of the associated $SO(1)/U(1)$-bundle over $M$. But $SO(1)/U(1)$ is a point, and thus the space of sections is contractible.
\end{remark}

To introduce the homological obstruction of dividing sets, given a $b$-plane field $\eta$ in $(M,\hat Z)$, let $Z$ be a connected component of $\hat Z$ and denote by $e_Z(\eta)\in H^2(Z,\mathbb{Z})$ the Euler class of the vector-bundle over $Z$ with fiber $\eta$. Equivalently, this coincides with the Euler class of $\eta$ in $M$ restricted to $Z$. Recall that if $\eta$ is a $b$-plane field in $(M,Z)$, the associated Euler class $e_Z(\eta)$ defined by $\eta$ on $Z$ is a homotopy invariant. Having fixed an orientation of $Z$ (induced by the orientation on $M$) and a fundamental class $[Z]\in H_2(Z) \subset H_2(M,\mathbb{Z})$,  the Euler class $e_Z(\eta)$ is determined by the integer $\langle e_Z(\eta), [Z]\rangle$. {Given a set of dividing closed curves on $Z$, recall that by convention (see the introduction) it comes with a choice of the decomposition $Z=Z_+\cup Z_-$ defined as follows. A defining function for the set of dividing curves determines $Z_+$ (respectively $Z_-$) as the components of $Z \setminus \Gamma$ where $f$ is non-negative (respectively non-positive).}

\begin{lemma}\label{prop:euler}
    Let $\xi$ be a $b$-contact structure in $(M,\hat Z)$. Let $Z$ be a connected component of $\hat Z$ and $\Gamma$ its dividing set. Then
    $$\chi(Z_+)-\chi(Z_-)=\langle e_Z(\eta), Z\rangle.$$
\end{lemma}

 This is the $b$-version of \cite[Lemma 9.2]{Honda}. 

\begin{proof}
 Let $\alpha=u\frac{dz}{z}+$ be a $\mathbb{R}$-invariant $b$-contact form defining $\xi$ in a neighborhood $U\cong Z\times (-\varepsilon,\varepsilon)$, which exists by virtue of Proposition \ref{prop:bconvex}. Now the evaluation $\langle e_Z(\eta),Z\rangle$ does not depend on the choice of representative of the class of $Z$ in $H_2(M,\mathbb{Z})$, so we can choose instead a parallel copy $Z^*=Z\times \{\delta\}$ of $Z$ completely contained in $Z\times (0,\varepsilon)$, where $\alpha$ defines a positive contact form. Then $\langle e_Z(\eta),Z\rangle= \langle e_Z(\eta),Z^*\rangle$. Since $Z^*$ is convex in the usual sense of contact geometry (as argued in the proof of Corollary \ref{cor:bgirouxcrit}) and its dividing set (in the sense of contact geometry) is $\Gamma\times \{\delta\}$, it follows from e.g. from \cite[Lemma 9.2]{Honda} (see also \cite[Lemma 4.1]{CO} and its discussion) that
$\langle e_Z(\eta),Z^*\rangle=\chi(Z_+)-\chi(Z_-),$
concluding the proof.
\end{proof}

\subsection{Fully overtwisted $b$-contact manifolds}

In this section, we establish an existence and classification theorem for $b$-contact structures that are {fully} overtwisted (i.e. contain overtwisted disks on each component of $M\setminus Z$), on closed $b$-manifolds of dimension three. {For the next statement, let us mention that if two $b$-contact structures $\xi_1,\xi_2$ in $(M,Z)$ have isotopic dividing sets $\Gamma_1, \Gamma_2$ along $Z$, then up to applying any $b$-isotopy that sends $\Gamma_1$ to $\Gamma_2$ and a second  $b$-isotopy using Proposition \ref{cor:dividingset},  we can assume that the image $\xi_1'$ of $\xi_1$ by this $b$-isotopic coincides with $\xi_2$ in a neighborhood of $Z$. In this situation, one can require (as a hypothesis) that $\xi_1'$ and $\xi_2$ are homotopic as $b$-plane fields relative to a neighborhood of $Z$. Notice that this will not depend on the $b$-isotopy that we used to identify the dividing sets and the neighborhoods of $Z$, and thus this is a hypothesis that only depends on $\xi_1$ and $\xi_2$.  When this is satisfied, we will say with a slight abuse of notation that $\xi_1, \xi_2$ have isotopic dividing sets and are homotopic as plane fields relative to $Z$.}

\begin{theorem}\label{thm:existencebcontacthprin}
Let $(M,Z)$ be a closed oriented $b$-manifold of dimension $3$  with $Z$ compact and let $\eta$ be a $b$-plane field. Then $\eta$ is homotopic through $b$-plane fields to a fully-overtwisted $b$-contact structure whose dividing set is any collection of separating curves in $Z$ such that $\chi(Z_+)-\chi(Z_-)=\langle e_Z(\eta), Z\rangle$. Furthermore, if two fully-overtwisted $b$-contact structures have isotopic dividing sets and are homotopic as plane fields relative to $Z$, then they are contact $b$-isotopic.
\end{theorem}

The idea of the proof is to construct a homotopy of $b$-plane fields in a neighborhood of $Z$ that turns $\eta$ into a $b$-contact structure whose dividing set is prescribed and satisfies the condition that $\chi(Z_+)-\chi(Z_-)=\langle e_Z(\eta), Z\rangle$. As we have seen in Lemma \ref{prop:euler}, this condition is necessary. Once this homotopy is done, the result follows from Eliashberg's relative $h$-principle. In particular, there are overtwisted disks contained in each connected component of $M\setminus Z$. Recall that such a $b$-contact structure is called \emph{fully overtwisted} (see Definition \ref{def:OT}).

We will now prove Theorem \ref{thm:existencebcontacthprin}.

\begin{proof}
As $\eta$ is a cooriented $b$-plane field, we can write $\eta= \ker \alpha$, where $\alpha \in {^b}\Omega^1(M)$ is non-vanishing. Let $z$ be a local defining function for $Z$ and $T$ be a tubular neighborhood around $Z$. Then $\alpha= u \frac{dz}{z}+\beta$, where $u \in C^\infty(T)$, $\beta\in \Omega^1(T)$ and $\iota_{\frac{\partial}{\partial z}}\beta=0$. 

\begin{lemma}\label{lem:bplane=plane}
Let $\Gamma$ be any collection of separating curves in $Z$ such that $\chi(Z_+)-\chi(Z_-)=\langle e_Z(\eta), Z\rangle$. In the tubular neighborhood $T$ around the critical set, there exists a homotopy $\alpha_t$ of $b$-forms such that
\begin{itemize}
    \item the kernel of $\alpha_t$ defines a $b$-plane field,
    \item $\alpha_0=\alpha$,
    \item  $\alpha_1$ is a convex $b$-contact form with dividing set $\Gamma$.
\end{itemize}
\end{lemma}

\begin{proof}
To prove the lemma, it is sufficient to find a smooth family $(u_s,\beta_s) \in C^\infty(T)\times \Omega^1(T)$ with $s\in I=[0,1]$, for each $s$ we have $\iota_{\pp{}{z}}\beta_s=0$, and such that
\begin{enumerate}
\item $(u_0,\beta_0)=(u,\beta)$,
\item $\{u_s=0\}\cap \{\beta_s=0\}=\emptyset$ for each $s\in I$,
\item $(u_1,\beta_1)$ are such that $u_1\in C^\infty(Z)$, $\beta_1\in \Omega^1(Z)$,
\item $u_1 d\beta_1+d\beta_1\wedge u_1 \neq 0$.
\item the zero set of $u_1$ is $\Gamma$. 
\end{enumerate}

Given such a family of functions and smooth forms on $T$, the family of $b$-form defined by $\alpha_t=u_t\frac{dz}{z}+\beta_t$ defines the desired homotopy: the first condition assures that $\alpha_0=\alpha$, the second one assures that $\ker \alpha_t$ is a $b$-plane field, and the last three ones that $\alpha_t$ is a convex $b$-contact form as stated in the lemma.

To prove the existence of $(u_s,\beta_s)$, the {observation is that} near $Z$, smooth plane fields and $b$-plane fields are in correspondence, and that the contact and $b$-contact condition coincide for $\mathbb{R}$-invariant structures. {This (non-canonical) identification works as follows:} We think of the pair $(u,\beta)$ as a one form $\tilde{\alpha}=udz+\beta$ such that $\tilde{\alpha}\neq 0$ defined on a neighborhood $T\cong Z\times (-\varepsilon,\varepsilon)$ of $Z$, where $z$ is as before a coordinate in the second component of $T$. The one form defined a plane field since $\eta$ was a $b$-plane field and so 
$$\{x\in T \enspace| \enspace u_x=0 \text{ and } \beta_x=0\} = \emptyset.$$

Given any set of separating curves $\Gamma$ in $Z$ such that $\chi(Z_+)-\chi(Z_-)=\langle e_Z(\eta),Z\rangle$, there exists a Morse function $u_1$ such that $\Gamma=u_1^{-1}(0)$, which is positive in $Z_+$, negative in $Z_-$ and that does not have any minimum with positive value nor a maximum with negative value. This is easily seen by considering an embedding of $e:\Gamma \rightarrow \mathbb{R}^3$, with coordinates $(x_1,x_2,x_2)$ such that
\begin{itemize}
    \item $u_1=e^*x_3$ is Morse,
    \item $e(\Sigma)\cap \{x_3=0\}=\Gamma$,
    \item $e(\Sigma)\cap\{x_3>0\}= Z_+$, $e(\Sigma)\cap \{x_3<0\},$
\end{itemize}
see for instance \cite[Fig. 4]{CO}. It then follows from \cite[Theorem 1.1]{CO} that there exists a $\mathbb{R}$-invariant contact form that writes $\tilde \alpha_1=u_1dz+\beta_1$. The Euler class of $\ker \tilde \alpha_1$ evaluated at $[Z]$ is given by  $\chi(Z_+)-\chi(Z_-)$ as per \cite[Lemma 9.2]{Honda} (or \cite[Lemma 4.1]{CO}). Recall that the homotopy class of a plane field defined near $Z$ is determined by its Euler class. Since $\tilde{\alpha}$ and $\tilde{\alpha}_1$ have the same Euler class, there is a homotopy of plane fields $\eta_t$ such that $\eta_0=\ker \tilde{\alpha}$ and $\eta_1=\ker \tilde{\alpha}_1$. This homotopy is generated by a smooth homotopy of defining one-forms $\lambda_t$, and we can easily assume that $\lambda_0=\tilde{\alpha}$ and $\lambda_1=\tilde{\alpha}_1$. Writing $\lambda_t$ as $u_tdz+\beta_t$,  the family $(u_t,\beta_t)$ satisfies $(1)$-$(5)$, and thus the family of $b$-forms given by $\alpha_t=u_t\frac{dz}{z}+\beta_t$ satisfies the statement of the Lemma.
\end{proof}
  
Given a set of separating curves in $Z$, we define the family of $b$-plane fields in $T$ given by the kernel of the $b$-forms
$$\overline{\alpha}_\tau=u_{\tau \cdot \varphi(z)}\frac{dz}{z}+\beta_{\tau \cdot \varphi(z)}, \enspace \tau \in [0,1]$$
where $u_s$ and $\beta_s$ are the families obtained in Lemma \ref{lem:bplane=plane}, and $\varphi:(-\epsilon,\epsilon)\to [0,1]$ is a smooth bump function such that
$$\begin{cases} \varphi(z)=0  \text{ for } z\in [\epsilon-\delta,\epsilon), \\
\varphi(z)=1 \text{ for } z \in [0,\delta),\\
\varphi(z)=\varphi(-z),
 \end{cases}$$
 for $\delta \in (0,\epsilon)$ small. The kernel of $\overline{\alpha}_\tau$ is clearly a $b$-plane field because $\{x\in M \enspace| \enspace (u_s)_x=0 \text{ and } (\beta_s)_x=0\} = \emptyset$ for all $s\in [0,1]$, so the same holds for $u_{\tau \cdot  \varphi(z)}, \beta_{\tau \cdot \varphi(z)}$ if we choose $\epsilon$ small enough.

 On the other hand $\overline{\alpha}_1$ is $b$-contact in $T'= Z\times (-\delta,\delta)$. Indeed, we have 
  $$\overline{\alpha}_1 \wedge d\overline{\alpha}_1 = \frac{dz}{z}\wedge (du_1\wedge \beta_1+u_1d\beta_1),$$
  which is non-vanishing in $T'$ by construction. Furthermore, the $b$ one-form $\overline{\alpha}_1$ agrees with $\alpha$ in $T\setminus T'$. Hence $\ker \overline{\alpha}_1$ extends to $M$ as $\ker\alpha$, is homotopic to $\ker \alpha$ and is of $b$-contact type near $Z$. To conclude the proof of the theorem, we apply the relative $h$-principle for overtwisted contact structures \cite{eli2} to homotope the plane field given by the restriction of $\ker \overline{\alpha}$ to $M\setminus T'$ to a contact structure that agrees with $\ker \overline{\alpha}$ near $T'$. We hence obtain a global $b$-contact structure on $M$ that is homotopic to $\eta$ and induces $\Gamma$ as the dividing set on $Z$. This proves the first part of the theorem. \\

  For the second part, we will use that the $h$-principle for overtwisted contact structures is parametric. Let $\xi_0, \xi_1$ be two $b$-contact structures on a closed $b$-manifold of dimension three $(M,Z)$, and assume they have isotopic dividing sets. By Proposition \ref{cor:dividingset}, the $b$-contact structures are $b$-isotopic on an neighborhood $U\cong Z\times (-\varepsilon,\varepsilon)$ of the critical set $Z$. By considering any global extension of this semi-local $b$-isotopy, we can assume that $\xi_0$ and $\xi_1$ coincide near $Z$. Let $N$ be any connected component of $M\setminus U'$, where $U'\cong Z\times (-\delta,\delta)$ for some $\delta<\varepsilon$. It is a connected three-dimensional manifold with boundary. By assumption, the $b$-contact structures $\xi_0, \xi_1$ induce on $N$ two contact structures that coincide near a neighborhood $V= (U\setminus U')\cap N$ of $\partial N$, they both admit some (possibly different) overtwisted disk in the interior of $N$, and $\xi_0, \xi_1$ are homotopic through formal contact structures relative to $V$ (i.e. through a homotopy fixed in $U'\cap N$. Recall that the set of embedded disks on $N$ is path-connected by compactly supported isotopies in its interior. Hence, if $\Delta_0$ and $\Delta_1$ are the embedded overtwisted disks in $N$ of $\xi_0$ and $\xi_1$ respectively, we can find a compactly supported isotopy of $N$ that sends $\Delta_0$ to $\Delta_1$. Hence, up to isotopy, we can assume that $\xi_0, \xi_1$ both admit a fixed overtwisted disk $\Delta \subset N \setminus V$. Observe that $N\setminus V$ is connected, and the fixed overtwisted disk $\Delta$ lies in the complement of $V$. This allows us to apply the parametric version of the $h$-principle for overtwisted contact structures relative to a domain \cite[Theorem 1.2]{bem} (the three-dimensional case as in \cite{eli2} is enough, but we refer to the very concrete statement in \cite{bem} for the sake of clarity) implies that there exists a path of contact structures $\xi_t$ with $t\in[0,1]$ such that $\xi_t|_{V}=\xi_0|_{V}$. We can apply this argument to each connected component of $M\setminus U'$, so in fact there is a global path of contact structure $\xi_t$ in $M\setminus U'$ that is constant near the boundary. This gives a global path of $b$-contact structures since we already argued that $\xi_0|_{U'}=\xi_1|_{U'}$. By Gray's stability for $b$-contact structures, we obtain a $b$-isotopy $\varphi_t : (M,Z) \longrightarrow (M,Z)$ such that $(\varphi_1)_*(\xi_0)=\xi_1$, as claimed.
\end{proof}

 \begin{remark}\label{rem:Nec}
 It follows from the discussion in the proof of Lemma \ref{lem:bplane=plane} that if there is $b$-contact structure $\xi$ on $(M,Z)$ inducing a dividing set $\Gamma$, then $\chi(Z_+)-\chi(Z_-)=\langle e_Z(\xi), Z \rangle$. Indeed, the nearby parallel copies of $Z$ are convex surfaces in the usual sense of contact geometry and have the same dividing set as $Z$. Then the equality above follows from the homotopical classification of convex surfaces in terms of its dividing set \cite{CO, Honda}. 
 \end{remark}

\subsection{Tight $b$-contact structures}
The existence of tight contact structures in three-manifolds is much more subtle than that of overtwisted ones: it was shown in \cite{EH} that there exist closed three-manifolds $M$ that do not admit tight contact structures. Using deep results known in contact topology, we will deduce that the existence of tight $b$-contact structures on a given $b$-manifold is an intricate question as well.

\begin{proposition}\label{prop:notight}
    There exists a $b$-manifold $(M,Z)$ that admits no tight $b$-contact structure.
\end{proposition}
\begin{proof}
    By the main result in \cite{EH}, there exists a closed three-manifold $M$ that admits no tight contact structure. Let $B$ be a ball embedded in $M$, and $Z\cong S^2$ its boundary. Assume that $(M,Z)$ admits a tight $b$-contact structure $\xi$, and we will reach a contradiction. First, a neighborhood of $Z$ has to be tight which by Corollary \ref{cor:bgirouxcrit} implies that the dividing set of $Z$ is connected. The connected component $V=B\setminus Z$ of $M\setminus Z$ is a three-ball, and by considering a $\mathbb{R}$-invariant $b$-contact form near $Z$, we can find a contact convex sphere in the contact manifold $(\tilde Z\subset V, \xi)$. Since each component of $M\setminus Z$ is a tight contact manifold by hypothesis, \cite[Theorem 2.1.3]{eli3} implies that the ball $U \subset B$ whose boundary is given by $\tilde Z$ equipped with $\xi|_U$ is contactomorphic to a standard Darboux ball (either positive or negative depending on the sign of the component $V$). We can find as well a sphere $Z'$ close to $Z$, contained in the other connected component of $M\setminus Z$ (i.e. the set $M\setminus B$) which is a (standard) convex sphere. Let $V'$ be the connected component of $M\setminus Z'$ that does not contain $B$. The manifold $(V',\xi|_{V'})$ is another tight contact manifold whose sign is the opposite of that of $V$. Since the boundary $Z'$ is a standard convex sphere, we can fill the complement of $V'$ by a standard Darboux ball of the same sign as $\xi|_{V'}$, obtaining a contact structure $\xi'$ in $M$ instead of a $b$-contact structure. \\
    
    We claim that $\xi'$ is tight. Indeed we know that $(V',\xi')$ is tight, and that $(\overline{M\setminus V'},\xi')$ is diffeomorphic to a standard Darboux ball. Let us show that $(M,\xi')$ is a tight contact manifold.
    
    Assume that $(M,\xi')$ is overtwisted. For any pair of Darboux balls, we know that there exists an isotopy of contactomorphisms of $(M,\xi')$ sending one to the other, see for example \cite[Theorem 2.6.7]{Ge}. Choose as one ball $M\setminus V'$, and as another ball a Darboux ball $B_2$  disjoint from an overtwisted disk in $M$. Since removing $B_2$ yields an overtwisted contact manifold with boundary, removing $M\setminus V'$ yields as well an overtwisted contact manifold. We know that this is not the case since $(V', \xi|_{V'})$ is tight, reaching a contradiction and proving that $(M,\xi')$ is a tight contact manifold. \\
    
     We have thus shown that if $(M,Z)$ admits a tight $b$-contact structure, then $M$ admits a tight contact structure. This is not the case, finishing the proof.
\end{proof}

Given a $3$-manifold $M$ that does not admit tight contact structure, we merely give the example of a hypersurface $Z$ (diffeomorphic to $S^2$) such that the $b$-manifold $(M,Z)$ does not admit any tight $b$-contact structure neither. It would be interesting to see if one could still construct a hypersurface on $M$ such that it does admit a tight $b$-contact structure.

\begin{remark}\label{rem:surgery}
    {Notice that to reach a contradiction in the proof of Proposition \ref{prop:notight}, we construct a contact structure on $M$ from a $b$-contact structure. More generally, this can be done under the following hypotheses. Let $\xi$ be a $b$-contact structure in $(M,Z)$, and decompose $M$ into $M=M_-\cup M_+$ where $M_-$ and $M_+$ are the closure of the components of $M\setminus Z$ where $\xi$ induces a negative (respectively positive) contact structure. We can think of $M_+$ and $M_-$ as contact manifolds with convex boundary (by removing a small neighborhood of $Z$ from them), and the boundary components of $M_+$ and $M_-$ are paired according to the component of $Z$ that they correspond to. Each such pair has the same dividing set. If there exists a positive contact structure $\xi'$ in $M_-$ that is convex along the boundary and has the same dividing set as the boundary of $M_+$, then we can replace $(M_-,\xi)$ by $(M_-,\xi')$ and glue it to $(M_+,\xi)$ to obtain a (genuine) contact structure in $M$. In the proof above, the manifold $(M_-,\xi')$ is a standard Darboux ball.}
\end{remark}

\section{Classifications up to $b$-isotopy}\label{sec:class}
Having dealt with the existence results for $b$-contact structures in the last sections, we now turn our attention to the classification of $b$-contact structures. More precisely, given a $b$-manifold $(M,Z)$, we study how many $b$-contact structures there are equivalent up to $b$-isotopy. As $M\setminus \mathcal{N}(Z)$ consists of a union of contact manifolds with boundary, the existing classification results for tight contact manifolds with boundary are relevant for the classification. The most well-known result is the following:

\begin{theorem}[Theorem 2.1.3 in \cite{eli3}]\label{thm:classificationball}
Two tight contact structures on the $3$-ball that coincide at the boundary are isotopic (relative to the boundary).
\end{theorem}

A more recent result was obtained for the solid torus, building on previous results \cite{H1}. A dividing set in the boundary of a tight solid torus is parametrized by $(n,-p,q)$, where $0 < q \leq p$ and $\operatorname{gcd}(p, q) = 1$. It means that having fixed a longitude and a meridian of the boundary of the solid torus, the dividing set is given by $2n$ curves whose homology class is given by $p$ times the longitude $S^1\times \{0\}$ of the solid torus and $q$ times the meridian $\{0\}\times \partial D^2$ of the solid torus. Define the integer numbers $r,s$ to be equal to $1$ if $(p,q)=1$, and otherwise, they depend on the continued fraction expansion of $-p/q$ as follows. Let $r_0,r_1,...,r_k$ denote the coefficients of the (negative) continued fraction expansion of $-p/q$, then 
\begin{align*}
r&=|(r_0+1)(r_1+1)...(r_{k-1}+1)r_k |,\\ s&=|(r_0+1)(r_1+1)...(r_{k-1}+1)(r_k+1)|.
\end{align*}
Then the first classification in \cite{H1} can be improved to an enumeration of the isotopy classes.
\begin{theorem}[\cite{LZ}]
Let  $M = S^1 \times D^2$ be a solid torus. Let $\Gamma$ be a dividing set on $\partial M$ parametrized by $(n, -p, q)$, where $0 < q \leq p$ and $\operatorname{gcd}(p, q) = 1$. Let the pair of integers $(r, s)$ be defined as above. Then the number of isotopy classes of tight contact structures on M with dividing set $\Gamma = (n, -p, q)$ is precisely
$$N(n, -p, q) = C_n((r - s)n + s),$$
where $C_n$ is the $n$-th Catalan number $\frac{1}{n+1} {2n \choose n}$.
\end{theorem}

\subsection{The case of $(S^3,S^2)$}
We will start by classifying $b$-contact structures in the prototypical example of a three-dimensional $b$-manifold: the three-sphere with critical set a two-sphere (say, the equator). We assume that the $b$-manifold $(S^3,S^2)$ is endowed with an orientation, hence there is a positive and a negative hemisphere. Let us give an example of a tight $b$-contact structure on $(S^3,S^2)$.

\begin{example}\label{ex:S3tight}
Let $(S^3, S^2)$ be the $b$-manifold described above. Choose a neighborhood $U=S^2\times (-\varepsilon, \varepsilon)$ of the critical set of the $b$-manifold, and denote by $z$ the coordinate in the second factor. Consider $udz+\beta$, where $u,\beta \in \Omega^*(S^2)$, be a $\mathbb{R}$-invariant contact form whose diving set along $S^2$ is just a circle (an equator). Such form exists, just consider a convex neighborhood of an embedded sphere {in a Darboux neighborhood}. By Giroux criteria, it will have a connected dividing set.

We can endow $U$ with the $b$-contact form $\alpha=u\frac{dz}{z}+ \beta$. Given some positive $\delta < \varepsilon$, the complement of $S^2\times (-\delta,\delta)$ is given by two disks $D_+,D_-$. In $D_+$, the orientation induced by the $b$-contact volume form defined by $\alpha$ near the boundary of $D_+$ induces the standard orientation of $S^3$. We can fill $D_+$ with the standard contact structure in the disk with convex boundary. In $D_-$, the $b$-contact volume form defined by $\alpha$ induces the opposite orientation. We can fill $D_-$ with the (negative) standard contact structure in the disk with convex boundary. This globally defines a tight $b$-contact structure in $(S^3,S^2)$. As we will see in this section, it is the unique (up to $b$-isotopy) positive tight $b$-contact structure in this $b$-manifold. As a $b$-Jacobi structure, the singular foliation has two three-dimensional leaves (the upper and lower hemisphere of $S^3$), two two-dimensional leaves (the two hemispheres on the critical set $S^2$), and a single one-dimensional leaf (the connected dividing set given by the equator on the critical set $S^2$).
\end{example}

As an intermediate step, we characterize the isotopy classes of {collections of curves} in $S^2$ that can arise as dividing sets of $b$-contact structures in $(S^3,S^2)$. Recall that given a set of separating closed curves in a surface $Z$, we denote by $Z_+,Z_-$ the (possibly disconnected) surfaces with boundary obtained by cutting open $Z$ along the curves.

\begin{lemma}\label{lem:tree}
Isotopy classes of closed curves that are the dividing set of some $b$-contact structure in $(S^3,S^2)$ satisfy $\chi(Z_+)=\chi(Z-)$ and are in one-to-one correspondence with two-equicolored trees with even number of vertices.
\end{lemma}

\begin{proof}
Given a set of curves $\Gamma$ that arise as the dividing set of a $b$-contact structure in $(S^3,S^2)$, the associated tree is defined as follows: each connected component of $S^2\setminus \Gamma$ defines a vertex and two vertices are related by an edge if they share a connected component of $\Gamma$ as their boundary. The resulting graph is a tree as {every curve in the two-sphere disconnects}. Recall that to each connected component of $S^2\setminus \Gamma$, there is an associated sign, and thus the tree is two-colorable. We denote the set of points of $S^2 \setminus \Gamma$ where the attributed sign is positive/negative by $S^2_+$, and respectively $S^2_-$.
{Since $H^2(S^3, \mathbb{Z})=0$, the Euler class of $\xi$ is necessarily trivial both in $S^3$ and when restricted to $Z=S^2$, and thus $\chi(S^2_+)=\chi(S^2_-)$.} In the associated tree, this means that there is an equal number of positive and negative vertices and therefore the tree is a two-equicolored tree. 

Conversely, given any two-equicolored tree, construct a set of separating curves $\Gamma$ in $S^2$ inducing this tree by the previous construction. Let $\eta$ be any $b$-plane field in $(S^3,S^2)$, its Euler class along $S^2$ is necessarily vanishing since the $b$-plane field extends to the disk defined by any of the two hemispheres of $S^3$. By Theorem \ref{thm:existencebcontacthprin}, there exists some $b$-contact structure whose dividing set is $\Gamma$.
\end{proof}

\begin{remark}
For $n=1,2,3,4,5$, the number of two-equicolored trees with $2n$ vertices is given by $1, 1, 4, 14, 65$, see \cite{tree}.
\end{remark}

We proceed now with the classification of $b$-contact structures in $(S^3,S^2)$.

\begin{theorem} \label{thm:bclassificationS3S2}
The classification of positive $b$-contact structures in $(S^3,S^2)$, endowed with some orientation, up to $b$-isotopy is the following.
\begin{itemize}
\item  There exists only one tight $b$-contact structure on $(S^3,S^2)$. {O}n the critical surface there is a single one-dimensional leaf.
\item {$b$-contact structures that are overtwisted in one connected component of $S^3\setminus S^2$ but tight in the other one are determined by specifying which side of the equator is the tight component (i.e. a sign), and by the homotopy class of the $b$-plane field in the closure of the overtwisted component relative to $S^2$ (i.e. by an element in $\mathbb{Z}$: the Hopf invariant)}. On the critical surface, there is a single one-dimensional leaf.
\item  For each $\Gamma$ isotopy class of separating closed curves in $S^2$ satisfying $\chi(S^2_+)=\chi(S^2_-)$, fully overtwisted $b$-contact structures inducing $\Gamma$ as the set of one-dimensional determined by {the homotopy class of the $b$-plane field in the closure of each component of $S^3\setminus S^2$ relative to $S^2$ (i.e. by an element in $\mathbb{Z}^2$: the Hopf invariant of the $b$-plane field in each component)}.
\end{itemize}
\end{theorem}

Note that the dividing curves that can be realized in the case where both connected components in $S^3\setminus S^2$ are overtwisted can be combinatorically characterized, see Lemma \ref{lem:tree}. 

\begin{remark}
By the correspondence between $b$-Jacobi structures and $b$-contact forms (Proposition \ref{prop:jacobi}) and Remark \ref{rem:definingform}, Theorem \ref{thm:classificationball} can be reinterpreted as the classification of $b$-Jacobi structures.
\end{remark}

\begin{proof}
By Proposition \ref{prop:bconvex}, we can always assume that the $b$-contact structure is $\mathbb{R}$-invariant close to $Z=S^2$. Furthermore, any isotopy class of separating curves that can arise as the set of one-dimensional leaves of a $b$-contact structure satisfies $\chi(S^2_+)=\chi(S^2_-)$ by Lemma \ref{lem:tree}, and uniquely determines the $b$-contact structure near $Z$ by Proposition \ref{cor:dividingset}. We will proceed by treating first the case where the $b$-contact structure is tight, followed by the case where one of the connected components admits an overtwisted disk. Lastly, we consider the case where the structure is fully overtwisted.

Assume first that the $b$-contact structure is tight.  By the Giroux criteria for $b$-contact structures (Corollary \ref{cor:bgirouxcrit}), the dividing set $\Gamma\subset Z=S^2$ has one connected component. Let $U= Z\times (-\varepsilon,\varepsilon)$ be an $\mathbb{R}$-invariant neighborhood of $Z$. Take $\delta>0$ smaller than $\varepsilon$, and consider $Z_{\delta}:=Z\times (-\delta,\delta)\subset Z\times (-\varepsilon,\varepsilon)$. Each connected component of $M\setminus Z_{\delta}$ is diffeomorphic to a closed disk endowed with a contact structure whose boundary is, by Lemma \ref{lem:convexdelta}, a convex sphere with a connected dividing set. The orientation induced by the contact structure coincides with the orientation of $S^3$ in one of these components, and with the opposite orientation on the other side. By Eliashberg's theorem (Theorem \ref{thm:classificationball}), there exists only one tight contact structure in the disk inducing in the boundary the convex surface characterized by a connected dividing set. Hence there is only one tight $b$-contact {structure} in $(S^3,S^2)$ up to isotopy: Example \ref{ex:S3tight}.

Assume now that one of the connected components is tight and the other one if overtwisted. Then by the Giroux criteria, the neighborhood around the critical set is tight. As before, the connected component where the $b$-contact structure is tight is uniquely determined by Eliashberg's theorem (Theorem \ref{thm:classificationball}). By the classification of overtwisted contact structures \cite{eli2}, every homotopy class of plane fields in the other connected component, fixed near the boundary, can be homotoped to a (unique) overtwisted contact structure. On a disk, the homotopy type of a plane field (relative to the boundary) is determined by an element in $\mathbb{Z}$, the Hopf invariant (see e.g. \cite[p. 142]{Ge}). Thus, there are $\mathbb{Z}$ overtwisted contact structures in that connected component. There is a choice of which connected component is the tight one, so there exist $\{+,-\}\times \mathbb{Z}$ different $b$-contact structures that are tight in one connected component and overtwisted in the other one. 

In the last case, when the $b$-contact structure is fully overtwisted, the dividing set determines the $b$-contact structure in a neighborhood around the critical set up to $b$-isotopy (Proposition \ref{cor:dividingset}). This dividing set is determined by a two-equicolored tree by Lemma \ref{lem:tree}. Each hemisphere can be filled by as many pair-wise distinct overtwisted contact structures as homotopy classes of plane fields (relative to the boundary) by Theorem \ref{thm:existencebcontacthprin}. Hence, there are $\mathbb{Z}\oplus \mathbb{Z}$ different overtwisted $b$-contact structures inducing that given dividing set, which concludes the proof of the theorem.
\end{proof}

\subsection{The case of $(S^3,\mathbb{T}^2)$}

We look at $S^3$ as the manifold obtained by gluing two solid tori along their boundary, by a diffeomorphism that glues the longitude of one boundary torus to the meridian of the other boundary torus. Let $\mathbb{T}^2$ denote the gluing locus, and we consider the $b$-manifold $(S^3, \mathbb{T}^2)$ endowed with an orientation (i.e. a sign on each solid torus). We can think as well of the critical surface as the boundary of a tubular neighborhood of an unknot on the three-sphere.

Given any $b$-contact structure on $(S^3,\mathbb{T}^2)$, the Euler class of $\xi$ is always zero, since $H^2(S^3)=0$. So when evaluated in $\mathbb{T}^2$, it is also zero. This shows that any dividing set of $b$-contact structure in $(S^3,\mathbb{T}^2)$ is a family of embedded close curves such that $\chi(\mathbb{T}^2_+)=\chi(\mathbb{T}^2_-)$. Arguing as in Lemma \ref{lem:tree}, these families of curves can also be characterized in terms of graphs (and a pair of coprime integers for some of the graphs). Indeed, any set of curves satisfying $\chi(T^2_+)=\chi(T^2_-)$ can be identified with a graph that might have a cycle if there are curves with non-trivial homology in the torus. When there is a cycle, the set of separating closed curves is determined by specifying the homology class of these curves. 

\begin{lemma}\label{lem:T2div}
Isotopy classes of separating curves that are the dividing set of some $b$-contact structure in $(S^3,\mathbb{T}^2)$ satisfy $\chi(\mathbb{T}^2_+)=\chi(\mathbb{T}^2_-)$ and are in one-to-one correspondence with pairs $(G, (p,q))$, where $G$ is a two-equicolored graph with an even number of vertices and at most one cycle, $(p,q)=(0,0)$ if $G$ has no cycle and $(p,q)$ are two coprime integers such that $0<q\leq p$ if $G$ has a cycle.
\end{lemma}

We can now state the classification of $b$-contact structures in $(S^3,\mathbb{T}^2)$.

\begin{theorem}\label{thm:classS3T2}
The classification of positive $b$-contact structures in $(S^3,\mathbb{T}^2)$, endowed with some orientation, up to $b$-isotopy is the following, depending on the dividing set along the critical surface.
\begin{itemize}
\item  For each isotopy class of separating curves $\Gamma$ in $\mathbb{T}^2$ such that $\chi(\mathbb{T}^2_+)=\chi(\mathbb{T}^2_-)$, fully overtwisted $b$-contact structures are {determined by the homotopy class of the $b$-plane on the closure of each overtwisted component relative to $T^2$ (i.e. by an element in $\mathbb{Z}^2\oplus \mathbb{Z}^2$: half the Euler class and the Hopf invariant of the $b$-plane field in each component).}
\item For each isotopy class as above without contractible closed curves, there are $2\cdot N(n,-p,q)$ tight $b$-contact structures, where $(n,-p,q)$ parametrizes the dividing set.
\item For each isotopy class as above without contractible closed curves, $b$-contact structures that are tight in one component and overtwisted in the other are determined by {specifying which component of $S^3\setminus T^2$ is tight (i.e. by a sign)}, a number in $N(n,-p,q)$ {that specifies the dividing set}, and {by the homotopy class of the $b$-plane field on the closure of the overtwisted component relative to $T^2$ (i.e. an element in $\mathbb{Z}^2$: half the Euler class and the Hopf invariant of the $b$-plane field)}.
\end{itemize}
\end{theorem}
\begin{proof}

As before, by Proposition \ref{prop:bconvex}, we can always assume that any $b$-contact structure is $\mathbb{R}$-invariant close to $Z=\mathbb{T}^2$. Let $\Gamma$ be any isotopy class of separating curves in $\mathbb{T}^2$ that can be realized by a $b$-contact structure on $(S^3,\mathbb{T}^2)$, we know it satisfies $\chi(\mathbb{T}^2_+)=\chi(\mathbb{T}^2_-)$ by Lemma \ref{lem:T2div}. It uniquely determines the $b$-contact structure near $Z$ by Proposition \ref{cor:dividingset}. Let $U= Z\times (-\varepsilon,\varepsilon)$ be an $\mathbb{R}$-invariant neighborhood of $Z$. Take $\delta>0$ smaller than $\varepsilon$, and consider $Z_{\delta}:=Z\times (-\delta,\delta)\subset Z\times (-\varepsilon,\varepsilon)$. By Eliashberg's classification of overtwisted contact structures \cite{eli2}, there is a unique overtwisted contact structure in each relative homotopy class of plane fields in each component of $M\setminus Z_{\delta}$ whose boundary is a convex surface whose dividing set is $\Gamma$. Each component of $M\setminus Z_{\delta}$ is diffeomorphic to a solid torus with convex boundary and dividing set $\Gamma$ by Lemma \ref{lem:convexdelta}. The relative homotopy class of a plane field in $D^2\times S^1$ is determined by two integers, one is the Euler class and the other one is the Hopf invariant (see \cite[Section 4.2]{Ge}). These capture the obstructions for two plane fields to be homotopic relative to the boundary, and are classes in $H^2(S^1\times D^2, \mathbb{T}^2; \pi_2(S^2))\cong \mathbb{Z}$ and in $ H^3(S^1\times D^2, \mathbb{T}^2; \pi_3(S^2)) \cong \mathbb{Z}$. {For any even element in $H^2(S^1\times D^2, \mathbb{T}^2;\pi_2(S^2))$ and any element in $ H^3(S^1\times D^2, \mathbb{T}^2; \pi_3(S^2)) \cong \mathbb{Z}$, there is a $b$-plane field with those invariants (this is analogous to the arguments for plane fields in \cite[Section 4.2]{Ge} due to the isomorphism between the $b$-tangent bundle and the tangent bundle in our context).} Hence, fully overtwisted $b$-contact structures for each possible dividing set on $\mathbb{T}^2$ are {parametrized by two integers, the Euler class divided by two and the Hopf invariant}.

If the $b$-contact structure has some tight component, then by Giroux criteria (Corollary \ref{cor:bgirouxcrit}) the dividing set in the boundary of any component of $M\setminus Z_\delta$ has no contractible closed curves. Hence, this holds for the set of one-dimensional leaves of the $b$-contact structure along $Z$. Each component is a solid torus, so it can be filled with $N(n,-p,q)$ tight contact structures or $\mathbb{Z}^2$ overtwisted contact structures. Hence there are $2 \cdot N(n,-p,q)$ tight $b$-contact structures. A $b$-contact structure that is overtwisted in one component and tight in the other is determined by choosing which component is tight, a tight structure filling $D^2\times S^1$ and a homotopy class of plane fields on the solid torus (which determines a unique overtwisted contact structure). This completes the proof of the theorem.
\end{proof}

The strategy as employed in the proofs of Theorem \ref{thm:bclassificationS3S2} and \ref{thm:classS3T2} works in other examples as well. As already mentioned, we use the classification of the homotopy class of plane fields on each connected component of $M\setminus Z$ and the classification of tight contact structures on the connected components of $M\setminus Z$ (seen as a manifold with boundary), with a convex boundary $Z$. For example, one can classify following the same scheme of proof $b$-contact structures in $S^3$ with $Z$ a finite number of parallel copies of $S^2$.

\section{Towards $b$-contact topology in higher dimensions}\label{sec:Reeb}

In this section, we give an existence theorem for $b$-contact structures in higher dimensions and generalize results in \cite{MO2} on Reeb dynamics.

\subsection{Existence of fully overtwisted $b$-contact structures in all dimensions}

In Section \ref{sec:fot}, we gave an existence $h$-principle for $b$-contact structures in three dimensions, by further prescribing the singular foliation induced by the $b$-Jacobi structure. We will make use of the recent advances in higher dimensional convex hypersurface theory \cite{HH,EP} to give a weaker existence theorem in all dimensions. Note that in this case, we cannot prescribe the singular foliation as we did in the three-dimensional case.

\begin{theorem}\label{thm:HDcontactexist}
Let $(M,Z)$ be a closed oriented $b$-manifold of dimension $2n+1\geq 3$ such that $Z$ is separating. Then any formal $b$-contact structure is homotopic to a fully overtwisted $b$-contact structure.
\end{theorem}
\begin{proof}
The strategy of the proof is similar to that of Theorem \ref{thm:existencebcontacthprin}, but a complete homotopical understanding of convex hypersurface theory is not available as in the three-dimensional case. Let $(\alpha_0,\omega_0)$ be an formal $b$-contact structure. We write $\alpha_0$ and $\omega_0$ on a neighborhood $U=Z\times (-\varepsilon,\varepsilon)$ of $Z$ as
$$ \alpha_0=u_0\frac{dz}{z}+ \beta_0, \qquad \omega_0=\gamma_0\wedge \frac{dz}{z} + \mu_0,$$
where $u_0\in C^\infty(U)$, $\beta_0, \gamma_0 \in \Omega^1(U)$ and $\mu_0\in \Omega^2(U)$ are such that  $\iota_{\partial_z}\beta_0 =\iota_{\partial_z}\gamma_0=\iota_{\partial_z} \mu_0=0$. We identify sections $\bigwedge^k {}^bT^*M$ with sections of $\bigwedge^k T^*M$ in $U$ via the identification $\frac{dz}{z}$ with  $dz$ as at the beginning of the proof of Theorem \ref{thm:existencebcontacthprin}. Let $(\tilde \alpha_0,\tilde \omega_0)$ be the formal contact structure obtained via this identification, i.e. 
$$\tilde \alpha_0= u_0 dz + \beta_0, \qquad \tilde \omega_0= \gamma_0 \wedge dz +\mu_0.$$ By Gromov's $h$-principle \cite{G}, there exists a family of formal contact structures $(\tilde \alpha_t, \tilde \omega_t)$ with $t\in [0,1]$ such that for $t=1$ it defines a contact structure in $U$. Now using either \cite{HH} or \cite{EP}, the hypersurface $Z\subset U$ is isotopic to a $C^0$-close convex hypersurface $\tilde Z$. There is a family of embeddings of 
$$\varphi_t:Z\times (-\delta,\delta) \longrightarrow U,$$
 with $\delta<\varepsilon$, such that $\varphi_0$ is just the trivial inclusion of $Z\times (-\delta,\delta)$ in $U$ and such that $(\tilde \alpha_2,\tilde \omega_2)= \varphi_1^*(\tilde \alpha_1,\tilde \omega_1)$ is an $\mathbb{R}$-invariant contact structure in $Z\times (-\delta,\delta)\subset U$. Hence we obtained in $Z\times (-\delta,\delta)$ a family of formal contact structures $(\tilde \alpha_t, \tilde \omega_t)$ with $t\in [0,2]$ such that for $t=2$ we have an $\mathbb{R}$-invariant contact form. We can write this pair of forms as 
 $$\tilde \alpha_t= u_tdz + \beta_t, \quad \tilde \omega_t= \gamma_t\wedge dz + \mu_t,$$
 such that $\beta_t, \gamma_t, \mu_t$ have no $dz$ term. Furthermore, by the $\mathbb{R}$-invariance near $t=2$, we can assume that $\beta_2, \gamma_2, \mu_2 \in \Omega^*(Z)$. We can associate a family of pairs of $b$-forms $(\alpha_t,\omega_t)$ with $t\in [0,2]$ by the identification $dz \mapsto \frac{dz}{z}$. That is
$$\alpha_t=u_t\frac{dz}{z} + \beta_t, \qquad \omega_t= \gamma_t\wedge \frac{dz}{z} + \mu_t. $$
First, we claim this family is a family of formal $b$-contact structures. To see this, observe that requiring that $(\alpha_t,\omega_t)$ if a formal $b$-contact structure, i.e. that $\alpha_t \wedge (\omega_t)^n >0$,
 can be written in terms of the decomposition as
$$  \alpha_t \wedge (\omega_t)^n= \frac{dz}{z} \wedge \Big(  u_t (\mu_t)^n + n\beta_t \wedge \gamma_t \wedge (\mu_t)^{n-1} \Big)>0.$$
On the other hand, requiring that $(\tilde \alpha_t, \tilde \omega_t)$ is a formal contact structure writes 
$$  \tilde \alpha_t \wedge (\tilde \omega_t)^n= dz \wedge \Big(  u_t (\mu_t)^n + n\beta_t \wedge \gamma_t \wedge (\mu_t)^{n-1} \Big)>0.$$
Since by construction $\mu, \beta_t, \gamma_t$ have no $dz$-term, these conditions are equivalent and simply require that $u_t (\mu_t)^n + n\beta_t \wedge \gamma_t \wedge (\mu_t)^{n-1}$ is non-degenerate along $\ker dz$.

What is not clear a priori is whether $(\alpha_2, \omega_2)$ is a genuine $b$-contact form, even if $(\tilde \alpha_2, \tilde \omega_2)$ is a genuine contact form. This will be the case exactly because we chose the splitting to be $\mathbb{R}$-invariant for $t=2$. Indeed, recall that for $t=2$ the forms $\beta_2, \gamma_2, \mu_2$ lie in $\Omega^*(Z)$. Since $(\tilde \alpha_2, \tilde \omega_2)$ is a genuine contact form, we have $\tilde \omega_2=d\alpha_2$ and thus:
$$ \tilde \alpha_2= u_2 dz + \beta_2, \qquad \tilde \omega_2= du_2 \wedge dz + d\beta_2, $$
and hence $\gamma_2=du_2$ and $\mu_2=d\beta_2$. By the $\mathbb{R}$-invariance of the forms $u_2, \beta_2$, this immediately implies that 
$$ \omega_2=  du_2 \wedge \frac{dz}{z} + d\beta_2= d\alpha_2,$$
i.e. that the $b$-form $\omega_2$ exactly corresponds to the $b$-exterior derivative of $\alpha_2$. We have thus shown that $(\alpha_t,\omega_t)$ is a homotopy of formal $b$-contact structures that is genuine for $t=2$. This family does not extend globally immediately, but we just need to use the homotopy extension property for formal solutions \cite[13.2.1]{hprinc} to find a global homotopy of formal $b$-contact structures $(\hat \alpha_t,\hat \omega_t), t\in [0,2]$ on $M$ which is constant and equal to $(\alpha_0,\omega_0)$ in $M \setminus U'$ for some smaller neighborhood $U'\subset U$ of $Z$, and equal to $(\alpha_t,\omega_t)$ in an even smaller neighborhood $U''\subset U'$ of $Z$. The pair $(\hat \alpha_2, \hat \omega_2)$ is thus genuine in $U''$ and hence restricted to a connected component $N$ of $M\setminus Z$ it defines a formal contact structure that is genuine near $V=U''\cap N$. We can apply the existence $h$-principle for overtwisted contact structures on $N$ relative $U''\cap N$ using \cite[Theorem 1.1]{bem}. Applying it to each connected component of $M\setminus Z$, we obtain a global $b$-contact structure $(\alpha', \omega')$ that is formally homotopic to $(\alpha_0,\omega_0)$.
\end{proof}
It follows from our discussion of the existence of formal $b$-contact structures that $(M,Z)$ admits a $b$-contact structure if and only if $M$ is stable almost complex.

\begin{corollary}
Let $M$ be a stable almost complex odd-dimensional closed manifold. For any separating hypersurface $Z$, the $b$-manifold $(M,Z)$ admits a fully overtwisted $b$-contact manifold.
\end{corollary}

\subsection{Periodic orbits on the critical set}

Using the observation in Proposition \ref{prop:bconvex}, we can show that there exist always contact leaves on the critical set of a $b$-contact manifold. Furthermore, the $b$-Reeb field defined by any $b$-contact form is always tangent to such submanifolds, where it is parallel to the Reeb field defined by a contact form.

\begin{theorem}\label{prop:codim2contactleaf}
Let $\xi$ be a $b$-contact structure on $(M^{2n+1},Z)$. 
Then the following holds:
\begin{itemize}
    \item There exists a one-parametric smooth family {$\{W_\delta\}_{\delta\in (-\tau,\tau)}$} of codimension two submanifolds contained in $Z$ where $\xi$ induces a (regular) contact structure;
    \item for any $\alpha$ defining $\xi$ the $b$-Reeb field of $\alpha$ is tangent to each $W_{\delta}$, where it coincides with a reparametrization of the Reeb field defined by a (regular) contact form of $W_{\delta}$.
\end{itemize}
\end{theorem}

\begin{proof}
Even if the first point only depends on $\xi$ and not on a fixed $b$-contact form, we will fix an arbitrary $\alpha$ from the beginning, as this has to be done anyway to prove the second point in the statement.
Write $\alpha$ on a neighborhood $U=Z\times (-\varepsilon,\varepsilon)$ of $Z$ as $\alpha=u\frac{dz}{z}+\beta$ with $u\in C^\infty(U)$ and $\beta \in \Omega^1(U)$. Even if $\alpha$ is not convex, the $b$-contact condition still implies that on $Z$
$$u(d\beta)^n+n\beta\wedge du \wedge (d\beta)^{n-1} \neq 0,$$
as can be seen from Equation (\ref{eq:bcontact}).
 On the set $W=\{x\in Z \enspace | \enspace u(x)=0\}$, this implies that $du\neq 0$ and thus $W$ is a regular hypersurface in $Z$. Furthermore, this also implies that $i_W^* \alpha= \beta$ is a contact form, where $i_W$ denotes the inclusion of $W$ in $M$. The set $W$ is non-empty, since otherwise {$\left(d\left(\frac{\beta}{u}\right)\right)^n$}
 restricted to $Z$ defines an exact volume form. Finally, the level sets $W_{\delta}=\{x\in Z \enspace|\enspace u(x)=\delta\}$ are regular as well for every {$\delta\in (-\tau,\tau)$ with $\tau>0$ small enough,} and $i_{W_{\delta}}\alpha$ defines a contact form as well since being contact is an open condition. 
 
 It remains to show that the $b$-Reeb field defined by $\alpha$ is tangent to $W_{\delta}$ and coincides with the Reeb field defined by the contact form $\beta_{\delta}=i_{W_\delta}^*\alpha$. To see this, write the Reeb vector field of $\alpha$ in $U$ as $R_\alpha=g \cdot z\pp{}{z} + X$, where $g\in C^\infty(U)$ and $X$ is tangent to the hypersurfaces $Z\times \{z\}$ with $z\in (-\varepsilon, \varepsilon)$. Denote by $i:Z\longrightarrow U$ the inclusion of the critical set in $U$. The equations $\iota_{R_\alpha} \alpha=1$ and $\iota_{R_\alpha}d\alpha=0$ evaluated in $Z$ yield
 {$$\begin{cases} 
  g \cdot u +\beta(X)=1 \\
 gd_Zu+\iota_Xd_Z\beta + \iota_Xdu\wedge \frac{dz}{z} + z \cdot \iota_{R_\alpha}(\frac{dz}{z}\wedge \dot{\beta})=0,
 \end{cases}
 $$}
 At $z=0$, we deduce that $\iota_Xdu=0$ and that $\iota_Xd_Z\beta=-gd_Zu$. The first equation implies that $i^*u$ is a first integral of $R_\alpha|_Z$, hence $R_\alpha$ is tangent to each $W_{\delta}$. Finally, the second equation pulled back to $W_{\delta}$ yields $\iota_Xi_{W_\delta}^*d\beta=0$. Hence $R_\alpha|_Z$ lies in the kernel of $\beta_\delta=i_{W_\delta}^*\beta$ and so is parallel to the Reeb field of $\beta$. Indeed, $\beta_\delta(R_\alpha)=\beta_\delta(X)=1-g \cdot u$, which is close to $1$ when $u$ is close to zero.
\end{proof}

\begin{remark}
The contact {submanifold $\{u=0\}$} contained on the critical set of a $b$-contact manifold obtained in {Theorem} \ref{prop:codim2contactleaf} corresponds to an odd-dimensional leaf of the singular foliation of the Jacobi structure associated with the $b$-contact structure. {Each component of $\{u=0\}$, as well as the one-parametric family of contact submanifolds, is a tight contact manifold, see \cite[Lemma 6.5]{dPW}}. The complement in $Z$, that is $Z \setminus \Gamma$ are even-dimensional leaves of the Jacobi structure, with an induced structure that is locally conformally symplectic.
\end{remark}

It follows that as long as the Weinstein conjecture holds in dimension $\dim M-2$, any $b$-Reeb field has infinitely many periodic orbits on $Z$. Since the Weinstein conjecture is known in dimension three \cite{Tau}, we deduce the following results, which generalizes \cite[Proposition 6.1]{MO2} to dimension five.

\begin{corollary}
Let $(M,Z)$ be a five-dimensional $b$-manifold endowed with a $b$-contact form $\alpha$. The $b$-Reeb field defined by $\alpha$ admits infinitely many periodic orbits on the critical set.
\end{corollary}

\bibliographystyle{alpha}

\begin{thebibliography}{99}

\bibitem{BM} M. Bertelson, G. Meigniez. \emph{Conformal Symplectic Structures, Foliations and Contact Structures}. To appear in J. Symplectic Geom. (2024).

	\bibitem{bem} M. S. Borman, Y. Eliashberg, E. Murphy. \emph{Existence and classification of overtwisted contact structures in all dimensions}. Acta Math. 215.2 (2015), 281-361.
	
	\bibitem{Can} A. Cannas da Silva. \emph{Fold-forms for four-folds.} J. Symplectic Geom. 8.2 (2010), 189-203.

    \bibitem{cavalcanti} G. Cavalcanti, \emph{Examples and counter‐examples of log‐symplectic manifolds}, J. Topol. 10.1 (2017), 1-21.
    
	\bibitem{CGW}  A. Cannas da Silva, V. Guillemin, C. Woodward. \textit{On the unfolding of folded symplectic structures}. Math. Res. Lett. 7.1 (2000), 35-53.
    
	\bibitem{CM2} R. Cardona, E. Miranda. \emph{On the volume elements on a manifold with transverse zeroes}. Regul. Chaotic Dyn. 24.2 (2019), 187-197.
	
	\bibitem{CMP} R. Cardona, E. Miranda, D. Peralta-Salas. \emph{Euler flows and singular geometric structures}. Philos. Trans. Roy. Soc. A 377.2158 (2019), 20190034.	
	
	\bibitem{CO} R. Cardona, C. Oms, \emph{Morse functions and contact convex surfaces}. \emph{Morse functions and contact convex surfaces}. J. Geom. Phys. 191 (2023), 104886.
	
	\bibitem{CPP} R. Casals, A. del Pino, F. Presas. \emph{$h$-principle for contact foliations}. Int. Math. Res. Not. IMRN 2015.20 (2015), 10176-10207.
	
	\bibitem{DM} M. Datta, S. Mukherjee. \emph{On existence of regular Jacobi structures}. Geom. Dedicata 173.1 (2014), 215-225.
	
	\bibitem{eli} Y. Eliashberg. \emph{On singularities of folding type}. Math. USSR Izv. 4 1119 (1970), doi: 10.1070/IM1970v004n05ABEH000946.
	
	\bibitem{eli2} Y. Eliashberg. \emph{Classification of overtwisted contact structures on 3-manifolds}. Invent. Math. 98.3 (1989), 623–637. 
	
	\bibitem{eli3} Y. Eliashberg. \emph{Contact 3-manifolds twenty years since J. Martinet's work}. Ann. Inst. Fourier (Grenoble) 42.1-2 (1992), 165-192.

        \bibitem{hprinc} Y. Eliashberg, N. Mishachev. \emph{Introduction to the h-principle}. AMS, Providence, RI, 2002.
 
	\bibitem{EP} Y. Eliashberg, D. Pancholi. \emph{Honda-Huang's work on contact convexity revisited}. Arxiv preprint (2022) arXiv:2207.07185.
	
	\bibitem{EH}  J. Etnyre, K. Honda, \emph{On the nonexistence of tight contact structures}, Ann. Math. (2001), 749--766.

\bibitem{FF} R. L. Fernandes, P. Frejlich. \emph{An h-principle for symplectic foliations}. Int. Math. Res. Not. IMRN 2012.7 (2012), 1505-1518.	
	
	\bibitem{FMM} P. Frejlich, D. Martínez-Torres, E. Miranda. \emph{A note on symplectic topology of b-symplectic manifolds}. J. Symplectic Geom. 15.3 (2017), 719–739.
	
	\bibitem{Ge} H. Geiges. \emph{An Introduction to Contact Topology}. Cambridge Univ. Press, Cambridge, 2008.
	
	
	\bibitem{GT} F. Gironella, L. Toussaint. \emph{Existence of conformal symplectic foliations on closed manifolds}. To appear in Math. Ann. (2024).

\bibitem{G1} E. Giroux. \emph{Convexity in contact topology}. Comment. Math. Helv. 66.4 (1991), 637–677.

\bibitem{G2} E. Giroux. \emph{Structures de contact en dimension trois et bifurcations des feuilletages de surfaces}. Invent. Math. 141 (2000), 615–689.	
	
	\bibitem{G} M. Gromov. \emph{Partial Differential Relations}. Ergebnisse der Mathematik und ihrer Grenzgebiete 9. Berlin: Springer, 1986.
	
		\bibitem{GMP} V. Guillemin, E. Miranda, A. R. Pires. \emph{Symplectic and Poisson geometry on b-manifolds}. Adv. Math. 264 (2014): 864-896.
  
	\bibitem{GMW} V. Guillemin, E. Miranda, J. Weitsman. \emph{Desingularizing $b^m$-symplectic structures}, Int. Math. Res. Not. IMRN 2019.10 (2019), 2981-2998.
	
	\bibitem{H1}K. Honda. \emph{On the classification of tight contact structures. I}. Geom. Topol. 4 (2000), 309–368.
	
	\bibitem{Honda} K. Honda. \emph{Contact geometry}. Available at  \nolinkurl{https://www.math.ucla.edu/~honda/math599/notes.pdf}{https://www.math.ucla.edu/honda/math599/notes.pdf}
	
	\bibitem{HH} K. Honda, Y. Huang. \emph{Convex hypersurface theory in contact topology}. ArXiv preprint (2019) arXiv:1907.06025.
	
	\bibitem{Kir} A. Kirillov. \emph{Local Lie algebras}, Russian Math. Surveys 31.4 (1976), 57–76.

     \bibitem{klaasse} R. L. Klaasse, \emph{Geometric structures and Lie algebroids}, PhD Thesis (2017).
     
	\bibitem{Lu} R. Lutz. \emph{Sur la g\'eom\'etrie des structures de contact invariantes}. Ann. Inst. Fourier (Grenoble) 29.1 (1979), 283-306.
	
		\bibitem{LZ} Z. Li, J. J. Zhang. \emph{Classification of tight contact structures on a solid torus}. ArXiv preprint (2020) arXiv:2006.16461.
		
		\bibitem{Li} A. Lichnerowicz. \emph{Les vari\'et\'es de Jacobi et leurs alg\`ebres de Lie associ\'ees}. J. Math. Pures Appl. 57 (1978), 453-488.


\bibitem{Me} R. Melrose. \emph{The Atiyah-Patodi-singer index theorem}. AK Peters/CRC Press, 1993.

	\bibitem{McS} D. McDuff, D. Salamon. \emph{Introduction to symplectic topology}. Vol. 27. Oxford University Press, 2017.	
	
	\bibitem{marcut} I. Marcut, B. Osorno Torres \emph{On cohomological obstructions for the existence of log-symplectic structures}, J. Symplectic Geom. 12.4 (2014), 863-866.
	
	\bibitem{MO1} E. Miranda, C. Oms. \emph{Contact structures with singularities: From local to global}. J. Geom. Phys. 192 (2023), 104957.
	
	\bibitem{MO2} E. Miranda, C. Oms. \emph{The singular Weinstein conjecture}. Adv. Math. 389 (2021), 107925.
	
	\bibitem{MOP} E. Miranda, C. Oms, D. Peralta-Salas. \emph{On the singular Weinstein conjecture and the existence of escape orbits for $b$-Beltrami fields}. Commun. Contemp. Math. 24.7 (2022), 2150076.
	
	\bibitem{NT} R. Nest, B. Tsygan. \emph{Formal deformations of symplectic manifolds with boundary}. J. Reine Angew. Math. 481 (1996), 27–54.
	
	\bibitem{dPW} A. del Pino, A. Witte. \emph{Regularisation of Lie algebroids and Applications}. J. Geom. Phys. 194 (2023), 105023.
	
	\bibitem{tree} N. Pippenger, Enumeration of equicolorable trees, SIAM J. Discrete Math. 14 (2001), 93-115.
	
	\bibitem{R} O. Radko. \emph{A classification of topologically stable Poisson structures on a compact oriented surface}. J. Symplectic Geom. 1.3 (2002), 523-542.
	
	\bibitem{Tau} C. Taubes. \emph{The Seiberg-Witten equations and the Weinstein conjecture}. Geom. Topol. 11 (2007), 2117–2202.
	
	\bibitem{V} I. Vaisman. \emph{Jacobi manifolds}. Selected topics in Geometry and Mathematical Physics, Vol. 1, (2002), 81-100.
	
\end{thebibliography}

\end{document}